\documentclass{amsart}
\usepackage{ifthen}
\usepackage{enumerate}
\usepackage{color}
\usepackage{verbatim}
\numberwithin{equation}{section}
\def\epsilon{\varepsilon}
\def\Log{\rm{Log}}
\def\downto{\searrow}

\def\de{\delta}
\def\si{\sigma}
\def\ga{\gamma}
\def\lap{\Delta}
\def\ti{\tilde}

\def\ga{\gamma}

\def\beq{\begin{eqnarray}}
  \def\eeq{\end{eqnarray}}

\def\al{\alpha}
\def\ep{\epsilon}

\def\partt{\frac{\partial }{\partial t} }

\def\phi{\varphi}
\def\R{\mathbb R}
\def\C{\mathbb C}
\def\boundary{\partial}
\def\N{\mathbb N}

\def\part{\partial}

\def\curlS{\mathcal S}
 
\def\grad{\nabla}

\def\counterword#1{%
  \ifthenelse{\ref{#1}=1}{one}{}%
  \ifthenelse{\ref{#1}=2}{two}{}%
  \ifthenelse{\ref{#1}=3}{three}{}%
  \ifthenelse{\ref{#1}=4}{four}{}%
  \ifthenelse{\ref{#1}=5}{five}{}%
  \ifthenelse{\ref{#1}=6}{six}{}%
}

\newtheorem{theorem}{Theorem}[section]
\newtheorem{lemma}[theorem]{Lemma}
\newtheorem{corollary}[theorem]{Corollary}

\theoremstyle{definition}

\theoremstyle{remark}

\definecolor{alert}{rgb}{0.8,0,0.3}
\newcommand{\alert}[1]{%
	\marginpar{%
		\ifodd\value{page} \raggedright \else \raggedleft \fi
		\footnotesize{\textcolor{alert}{#1}}
	}
}

\begin{document}
\title[Local estimates and uniqueness for the bi-harmonic heat flow]{Some local estimates and a uniqueness result for the entire biharmonic heat equation}
\author{Miles Simon}
\address{Miles Simon:
Institut f\"ur Analysis und Numerik (IAN), Universit\"at Magdeburg, Universit\"atsplatz 2, 39106 Magdeburg, Germany}
\email{miles point simon at ovgu point de}
\author{Glen Wheeler}
\address{Glen Wheeler:
Institut f\"ur Analysis und Numerik (IAN), Universit\"at Magdeburg, Universit\"atsplatz 2, 39106 Magdeburg, Germany}
\curraddr{Institute for Mathematics and its Applications, University of Wollongong, Wollongong 2522, Australia}
\email{glenw@uow.edu.au}
\date{February 2013}

\dedicatory{}


\begin{abstract}
We consider smooth solutions to the biharmonic heat equation on
$\R^n\times [0,T] $ for
which the square of the
Laplacian at time $t$ is globally bounded from above by $k_0/t$  for some $k_0$ in
$\R^+$, for all $t \in [0,T]$.
We prove local, in space and time, estimates for such solutions. We
explain how these estimates imply uniqueness of smooth solutions in
this class.
\end{abstract}

\maketitle
\section{Introduction}
In this paper we prove local in space and time estimates for solutions  $u:\R^n\times[0,T]\to\R$ of the biharmonic heat
flow,
\begin{align} 
\partt u  = -\lap^2 u, \label{bi}
\end{align}
assuming that we have some global in time control on how the solution
behaves as $t\searrow 0$. 
This control takes the form
\begin{equation}
\label{EQgc}
t  |\lap u|^2(x,t)  \leq k_0 < \infty
\end{equation}
for $t \in [0,T]$ for all $x \in \R^n$ for some fixed $k_0 \in \R^+$.
This means that it is possible for $ |\lap u|^2(x,t)$ to approach infinity  as $t \searrow 0$, but if it does so, then
we have some control over the rate at which this occurs.
Here $\lap u$ refers to the spatial Lapacian of $u$, $\lap u (x,t) =
\sum_{i=1}^n \grad_i \grad_i u(x,t)$ where $\grad_i u(x,t)$ is the partial
derivative of $u$ with respect to $x_i$.

The growth condition \eqref{EQgc} is natural in the following senses.
It is scale invariant: see the
explanation of the scale invariance of $b$ just after the
the definition of \eqref{assum1} in Section \ref{STblowup}.
This behaviour also does occur in an asymptotic sense. That is, it is
possible to construct a solution  $ u \in C^{\infty}( \R^n \times (0,T))$ and to find points $x(t) \in
\R^n $ for all $t>0$ such that 
$|\lap u|^2 (x(t),t)   = \frac{k_0}{t}$ for some fixed $k_0 >0$ for
all  $t>0$. 
We also find points $y(t) \in \R^n$ for all $t>0$ such that
$(\lap^2 u) (y(t),t)   = \frac{k_0}{t}$ for some fixed $k_0 \neq 0$ for
all  $t>0$. 
That is, the speed of $u$ is not integrable in time.
See the example in Section \ref{STexample}.

Our first result is the following local estimate.
\begin{theorem}\label{mainintro1}
Let $u:\R^n \times [0,T] \to \R$, $T < \infty,$  be a smooth solution to \eqref{bi}
that satisfies
\begin{align} 
 | \lap u |^2(x,t)  \leq \frac{k_0}{t} \label{introA1}
\end{align} 
for some $k_0 \in \R$, for all $t \in [0,T]$, all $x \in \R^n$ and
\begin{align*}
\sup_{B_1(0)} \sum_{i=0}^{2n+2} |\grad^i u_0|^2 \leq k_1< \infty,
\end{align*}
for some $k_1 \in \R$, where $u_0(\cdot) := u(\cdot,0)$. 
Then there exists an $N = N(n,k_0,k_1) >0$ such that
\begin{align*} 
|\lap u|^2(x,t) \leq \frac{N}{(1-|x|)^4}
\end{align*} 
for all $x$, $t$ which satisfy $x \in \overline B_1(0)$, $\,(1-|x|)^4 \geq
Nt$, and $\,t \leq  \frac1N$ , $t \leq T$.
\end{theorem}
For $i \in \N$ in the above and all that follows, $\grad^i u(x,t)$ refers to
the full $i$-th spatial derivative, and $|\grad^i u(x,t)| $ the standard norm
thereof.  For example: $\grad^2 u(x,t) = (\grad_i \grad_j u(x,t))_{i,j \in \{1,
\ldots, n\} }$ and $|\grad^2 u(x,t)|^2 = \sum_{i,j = 1}^n |\grad_i \grad_j
u(x,t)|^2$. The operator $\lap^k$ is defined by $\lap^k = (\lap)^k$. 

If we have control on other derivatives as $t \searrow 0$ then we obtain
results on higher regularity.
\begin{theorem}\label{mainintro2}
Let $s \in \N$, $s \geq 2$ and $u:\R^n \times [0,T] \to \R$, $T < \infty$  be a smooth solution to \eqref{bi}
which satisfies
\begin{equation}
\label{introA1s}
 \left( |\grad^s u| + |\grad^{s-1} u|^{s/(s-1)} +  \ldots + | \grad u |^{s}
 \right)(x,t)
\leq \frac{k_0}{ t^{s/4}}
\end{equation}
for some $k_0 \in \R$, for all $t \in [0,T]$, all $x \in \R^n$ and
\begin{align*} 
\sup_{B_1(0)} \sum_{i=0}^{2n+s+1} |\grad^i u_0|^2 \leq k_1< \infty,
\end{align*}
for some $k_1 \in \R$.
Then there exists an $N = N(n,k_0,k_1,s) >0$ such that
\begin{align*} 
 \left( |\grad^s u| + |\grad^{s-1} u|^{s/(s-1)} +  \ldots + | \grad u |^{s}
 \right)(x,t)
\leq \frac{N}{(1-|x|)^{s}}
\end{align*} 
for all $x$, $t$ which satisfy $x \in \overline B_1(0)$, $\,(1-|x|)^4 \geq
Nt$, and $\,t \leq  \frac1N$, $t \leq T$.
\end{theorem}
In Section \ref{Tych} we construct an example of a solution to \eqref{bi} which
starts off indentically equal to zero, becomes immediately non-zero and is
smooth in space and time.  Solutions of this type for the heat equation are
known to exist and were constructed by Tychonoff, see \cite{Tych}.  In
particular, smooth solutions are not uniquely determined by their initial
values:  $u(\cdot,\cdot) = 0$ is also a solution.  If however we consider 
smooth solutions which satisfy \eqref{introA1} then the solution is uniquely
determined by it's initial value, as we show in Section \ref{uniqueness}.  The
theorem we prove is:
\begin{theorem}\label{uniquenessintro}
Let $u,v:\R^n \times [0,T] \to \R$, $T < \infty$ be smooth solutions to \eqref{bi}
which satisfy
\begin{align*} 
 |\lap v|^2(x,t) + |\lap u |^2(x,t) \leq \frac{k_0}{ t}
\end{align*} 
for some $k_0 \in \R$,
for all $t \in [0,T]$ and all $x \in \R^n$ and
\begin{align*} 
u_0(\cdot) = v_0(\cdot).
\end{align*}
Then  $u \equiv v$.
\end{theorem}
The uniqueness problem for the classical heat flow has a rich history.
In the setting where one assumes the solution is non-negative, D. Widder
established uniqueness  for the heat flow on $\R$ for solutions whose
initial value is zero, see Theorem 5 of \cite{W}.
His method relied upon a specific integral representation of the
solution, Theorem 4 of \cite{W}, which
is valid for non-negative solutions.

This proved to be readily generalisable and so influential as to be given its
own name: a uniquness theorem is of Widder-type if the only hypotheses are on
the geometry of the ambient space and that the solution be non-negative.
For example, Aronson \cite{Ar} proved that non-negative solutions to second
order linear equations of divergence form (the coefficents of the
operator being sufficiently regular) in $\R^n$ are uniquely determined by
their initial data: see Section 5 of that paper. 

Here we take an approach reminiscent of \cite{Si1} that
complements the existing literature. Our assumptions for Theorem
\ref{uniquenessintro} are  global but we do not require any non-negativity of
the solution.
Although the flow \eqref{bi} is higher-order, we are able to obtain our
estimates using pointwise assumptions, as opposed to integral conditions.

The paper is organised as follows.
Section \ref{STblowup} contains the proof of Theorem \ref{mainintro1} and
Theorem \ref{mainintro2}.
These proofs require some energy estimates for solutions to \eqref{bi}, which
is the subject of Section \ref{STaprioriestimates}.
Section \ref{STuniqueness} contains the proof of Theorem \ref{uniquenessintro},
and Section \ref{Tych} contains full details on the Tychonoff-type solutions
discussed above.
We present in Section \ref{STexample} details on the construction of the example
mentioned above which shows that control of the form \eqref{EQgc} is
natural.
We also show that the solution in this example has a speed which is not integrable.

Some of the estimates from Section \ref{STblowup} and Section
\ref{STaprioriestimates} rely on interpolation inequalities which are not
readily available in the current literature.  These interpolation inequalities
are proved in the Appendix.

\section{A Blowup argument}
\label{STblowup}

Let $u: \R^n \times  [0,T] \to \R$ be a
smooth solution to
 \eqref{bi}.
We consider the scale invariant quantity 
$e(u): B_1(0) \times [0,T] \to \R$ defined by
\begin{align*} 
e(u)(x,t):= d^4(x) |\lap u|^2(x,t),
\end{align*}
where $d(x):= (1- |x|)$ is the distance from the boundary
$\boundary B_1(0)= \{ x \in \R^n \,|\, |x| = 1 \}$ to the point $x$ in
$B_1(0)$. 
Note that $e = 0$ on $\partial B_1(0)$.
The function $e$ is scale invariant in the following sense: 
If we define $\ti u( \ti x, \ti t):= u(c \ti x, c^4
\ti t)-c_0,$ where $c_0$ is an arbitrary constant in $\R$,
 then $\ti u:  \R^n \times  [0,\ti T] \to \R$ is still a
smooth solution to \eqref{bi} and the quantity $\ti e( \ti x, \ti t)
:= \ti e(\ti u)(\ti x,\ti t) $
which is defined by
\begin{align*} 
\ti e(\ti u)(\ti x,\ti t):= {\ti d}^4(\ti x)
|\lap \ti u|^2(\ti x,\ti t) ,
\end{align*}
satisfies
\begin{align*} 
\ti e(\ti x,\ti t) :=  e(x,t),
 \end{align*}
where here
$ x:= c \ti x$, $ t:=  c^4
\ti t$ , $\ti T = \frac{ T}{c^4} $, and ${\ti d}(\ti x) := ( \frac 1 c -
| \ti x|)$  is the distance from $x \in B_{1/c}(0)$ to $ \partial B_{1/c}(0)$.
The scale invariance of $e$ can be seen as follows: 
\begin{align*}
 (\grad \ti u) (\ti x, \ti t) &= c (\grad  u) (x, t)\,,\ &&\mbox{and hence} \ 
 (\grad^k \ti u ) (\ti x, \ti t)  = c^k ( \grad^k  u ) (x, t)\,, 
\\
 \left(\frac{\partial{}}{{\partial  t}} \ti u \right)(\ti x, \ti t) &= c^4 \Big(\partt  u\Big) ( x,   t)\,,
\ &&\hskip-1cm\mbox{and hence}\ \left(\Big(\frac {\partial } {\partial  t}\Big)^k
\ti u\right) (\ti x, \ti t) = c^{4k} \left( \Big(\partt   \Big)^k u\right) ( x,
 t)\,,\\
 \ti d( \ti x) &=  \Big( \frac 1 c - | \ti x|\Big)
&&\\&=  \frac 1 c (  1 - | c \ti
x|)&&\\ &=  \frac 1 c (  1 - |x|) = \frac 1 c d(x) \,,
 &&\mbox{and hence}\ 
 {\ti d}^4( \ti x) =  \frac 1 {c^4} d^4(x)\,,
\end{align*}
where here $(\partt )^k$ refers to $k$ time derivatives, and
$\grad^ku$ refers to $k$ spatial derivatives, and we are assuming that
$k \in \N$ ($k \neq0$).
Therefore
\begin{align*}
|\lap  \ti u|^2(\ti x, \ti t){\ti d}^4(\ti x)  = |c^2 \lap  u|^2 (x,t)
  \frac 1 {c^4} d^4(x)  = |\lap  u|^2 (x,t) d^4(x).
 \end{align*}
and $e(x,t) =
\ti e(\ti x, \ti t)$ as claimed.
Note that 
\begin{align*}
x \in B_1(0)\,, \ d^4(x) \geq Nt\quad \Longleftrightarrow \quad
\ti x \in B_{1/c}(0)\,, \ {\ti d}^4(\ti x) \geq N \ti t
\end{align*}
in view of the definitions of the terms involved.

In the following, we will assume that 
\begin{equation}
\tag{\rm A1} b(u)(x,t) :=  t  |\lap u|^2(x,t)   \leq k_0 < \infty
 \ \mbox{ for all } x \in \R^n\,, t \in [0,T],  \label{assum1}
\end{equation}
for some fixed $k_0 \in \R^+$.
That is, the quantity \begin{equation}
 Q(x,t):= Q(u)(x,t) := |\lap u|^2(x,t) \label{Qdefn}
\end{equation}
 may approach infinity
as $t \searrow 0$, but it is only allowed to do so at a controlled,
but non-integrable rate.
Note that we have $e(x,t) = d^4(x) Q(x,t) $ and $b(x,t) = t Q(x,t)$.
The function $b(x,t)$ is also scale invariant in the following sense:
If we define $\ti u$, $\ti x$ and $\ti t $ as above, then 
$\ti b(\ti x,\ti t) := b(\ti u)(\ti x, \ti t) = b(x,t)$
and hence 
 $\ti b(\ti x,\ti t)\leq k_0 < \infty$ for all $\ti x \in
\R^n$ and all $\ti t \in [0,\ti T]$.
The scale invariance of $b$ may be verified with an  argument
similar to the one
we used above to show that $e$ is scale invariant.
We are interested in the local behaviour of solutions $u$ to \eqref{bi} which satisfy \eqref{assum1}.
In particular, if at time zero $u_0 = u(\cdot,0)$ satisfies
\begin{equation}
\tag{A2} \sup_{B_1(0)} \sum_{i=0}^{2n+2} |\grad^i u_0|^2 \leq k_1< \infty,
\label{assum2}
\end{equation}
for some fixed $k_1 \in \R^+$, then we show that the solution satisfies estimates on a smaller ball
for a short well-defined time interval.
The following theorem is Theorem \ref{mainintro1} of the introduction. 
\begin{theorem}\label{main1}
Let $u:\R^n \times [0,T] \to \R$, $T < \infty$  be a smooth solution to \eqref{bi}
which satisfies assumptions \eqref{assum1} and \eqref{assum2}.
Then there exists an $N = N(n,k_0,k_1) >0$ such that
\begin{align*}
e(x,t) \leq N
\end{align*}
for all $x$, $t$ which satisfy $x \in \overline B_1(0)$, $\,d^4(x)
\geq Nt$, and $\,t \leq  \frac1N$, $t \leq T$.
\end{theorem}
For our theorem on higher order regularity, we modify the quantities
above.
Let $s \in \N$, $s \geq 2$ be given and fixed, and define
\begin{align*}
& Q_s(u)(x,t) = ( |\grad^s u| + |\grad^{s-1} u|^{s/(s-1)} +
 \ldots + | \grad u |^{s} )^{4/s}(x,t)  \\
& e_s(u)(x,t) = d^4(x)  Q_s(u)(x,t) \\
& b_s(u)(x,t)  = t  Q_s(u)(x,t) .
\end{align*}
These quantities are scale invariant in the sense explained above:
for $\ti u$, $\ti T$, $\ti t$, $\ti x$ and $\ti d$   defined as above, and
$\ti Q_s(\ti x , \ti t) := Q_s(\ti u)(\ti x, \ti t)$ we have
\begin{align*}
 \ti e_s(\ti x,\ti t) &:= {\ti d}^4(\ti x)  \ti Q_s (\ti x,\ti t) =  d^4(x)  Q_s(u)(x,t)\,,\quad\text{and}\\
 \ti b_s(\ti x, \ti t)  &:= \ti t \ti Q_s (\ti x,\ti t)  =  t  Q_s (x, t) . 
\end{align*}
For this set-up we require
\begin{equation}
\tag{${\rm A}_{\rm s}1$}
\label{assumt1}
b_s(u)(x,t)  \leq k_0 < \infty 
\end{equation}
for all $x \in \R^n$ $t\in [0,T]$, and for some fixed $k_0\in\R^+$, $s \geq 2$,
$s \in \N$; and
\begin{align}
\tag{${\rm A}_{\rm s}2$}
\label{assumt2}
\sup_{B_1(0)} \sum_{i=0}^{2n+s+1} |\grad^i u_0|^2 \leq k_1< \infty,
\end{align}
for some fixed $k_1 \in \R^+$.
In this context we obtain the following variant of Theorem \ref{main1} above.
\begin{theorem}\label{main2}
Let $u:\R^n \times [0,T] \to \R$, $T < \infty$  be a smooth solution to \eqref{bi}
which satisfies assumptions ($A_s$1) for some $s \in N, s\geq 2$  and  ($A_s$2).
Then there exists an $N = N(n,k_0,k_1,s) >0$ such that
\begin{align}
e_s(x,t) \leq N
\,,\label{conclusion2}
\end{align}
for all $x$, $t$ which satisfy $x \in \overline B_1(0)$, $d^4(x) \geq Nt$, $t
\leq \frac1N$, $t \leq T$. \end{theorem}
Note that this theorem is equivalent to
Theorem \ref{mainintro2} of the introduction.
Under the same  assumptions as Theorem \ref{main2},  but with the condition that $s \geq 4$ , we also obtain  a local supremum
bound for $|u|$:
\begin{corollary}\label{main3}
Assume everything is as in Theorem \ref{main2} but that $s \geq 4$. Then we also have
\begin{align}
|u(x,t)| \leq \sqrt{ k_1 } + 1 \label{conclusion3}
\end{align}
for all $x$, $t$ which satisfy $x \in \overline B_1(0)$, $d^4(x) \geq Nt$, $t
\leq \frac1N$, where $N$ is as in the conclusion of Theorem \ref{main2}
above.
\end{corollary}
\begin{proof}[Proof of Corollary \ref{main3}.]
Let $(x,t_0)$ be a point which satisfies  $d^4(x) \geq Nt_0$ and $t_0  \leq
\frac 1 N$.
Then  $d^4(x) \geq Nt $ and $ t \leq \frac 1 N$ for all $t \leq t_0$.
Hence, taking $s=4$ in Theorem \ref{main2} we see that $|\lap^2 u |(x,t)| \leq
\frac{N}{d^4(x)}$ for all $t \leq t_0$ and that $|u(x,0)| \leq \sqrt{k_1}$ in view
of \eqref{assumt2}.
The evolution equation for $u(x,t)$ is
$\partt u(x,t) = -\lap^2 u (x,t)$. Integrating this from $0$ to $t_0$
and using the two estimates which we just derived, we see that
$ |u(x,t_0)| \leq t_0 \frac{N}{d^4(x)} +\sqrt{ k_1 }$. Using that $d^4(x) \geq
N t_0$
we obtain the result.
\end{proof}
 
Now we prove Theorem \ref{main1}.

\begin{proof}[Proof of Theorem \ref{main1}.]
Define $d$ resp.  $e$ to
be zero on $\R^n \backslash B_1(0)$ resp. $(\R^n \backslash B_1(0))
\times  [0,T]$. Then $e$ is continuous on $\R^n \times  [0,T]$.
Assume that the conclusion of the theorem is false, and let $N \in
\N$.
 Note that $e(x,0) \leq k_0$ where $k_0$ is the constant appearing in
 \eqref{assum2}. Without loss of generality $N >k_0$.
The set of $x \in \overline{B_1(0)} $,  $ t\in [0,T]$ for which
$1 \geq d^4(x) \geq Nt$ and  $t
\leq  \frac 1 N$ is a compact set in $\R^n \times  [0,T]$  which we
denote by $K$.
By compactness of $K$ and continuity of $e$ and the fact that $e(x,0)
\leq k_0 < N$ for all $x \in \overline{B_1(0)}$ , there must be a first time $t_0 \in (0,\frac 1 N]$ and (at least) one point
$x_0 \in B_1(0)$ such that $e(x_0,t_0) = N$. That is:
$e(x,t) < N$ for all $(x,t) \in K$ with  $ t < t_0$, and 
$e(x_0,t_0) = N$ for some point $(x_0,t_0) \in K$. Clearly we have
$d(x_0) >0$ for such a point, that is, $x_0 \in B_1(0)$, since $e(x_0,t_0)
> 0$.
Rescale the solution $u$ to $\ti u( \ti x, \ti t):= 
u(c \ti x, c^4 \ti t) -c_0$, where $c_0:= u(c \ti x_0,0)$, and  $c>0$
is chosen so that   ${\ti d}^4(\ti
x_0) = N$.
It  is possible to choose $c$ in this way: $\ti d(\ti x) = \frac 1 c d(x)$, so we choose $c^4 =
\frac{d^4(x_0)}{N}$, which is larger than zero since $d(x_0) > 0$ as we
explained above. 
Our choice of $c_0 $ guarantees that $\ti u(\ti x_0,0) = 0$.
Note for later use that $c^4 = \frac{d^4(x_0)}{N} \leq \frac 1 N\,  (\leq
1)$, and $c \searrow 0$ as $ N \to \infty$.
Now
\begin{align*}
N  = e(x_0,t_0)  = \ti e(\ti x_0,\ti t_0) 
 = {\ti d}^4(\ti x_0) \ti Q(\ti
x_0,\ti t_0)  = N \ti Q(\ti x_0,\ti t_0)
\end{align*}
due to scaling, and hence
\begin{align*}
\ti Q(\ti x_0,\ti t_0) = 1.
\end{align*}
Similarly, 
\begin{align*}
N \geq  e(x,t) 
 = \ti e(\ti x,\ti t) 
 = {\ti d}^4(\ti x) \ti Q(\ti
x,\ti t) 
\end{align*}
for all $(x,t) \in K$ with $t \leq t_0$,  implies  
\begin{align}
\ti Q( \ti x , \ti t) \leq  \frac N {  {\ti d}^4(\ti x)  } \label{mainest}
\end{align}
 for all $\ti x \in B_{1/c}(0)$ with ${\ti d}^4(\ti x) \geq \ti t N$ and
 $\ti t \leq \ti t_0$.
Note that the inequality \eqref{mainest} is also valid for 
all $\ti x $ with ${\ti d}^4(\ti x) \geq \ti t N$ and
 $\ti t \leq \ti t_0$, since $\ti d (\ti x) = 0$ outside of $
 B_{1/c}(0)$ (here we define $\frac{M}{0} = \infty$ for $M>0$).
As in the paper \cite{Si1} we consider two cases:
{\bf Case 1}, where ${\ti d}^4(\ti x_0) \geq 2N\ti t_0$ (which  is
equivalent to $\ti t_0 \leq \frac 1 2$, since ${\ti d}^4(\ti x_0) = N$),  and
{\bf Case 2}, where ${\ti d}^4(\ti x_0) < 2N\ti t_0$ (which  is
equivalent to $1 \geq \ti t_0 > \frac 1 2$, since $\ti t_0 N \leq {\ti d}^4(\ti x_0) <2N\ti t_0$ and ${\ti d}^4(\ti x_0 ) = N$.
Note that $N \ti t_0 \leq {\ti d}^4(\ti x_0) $ since $(x_0,t_0) \in K$
).
We start with Case 1.

\vspace*{5mm}
\noindent {\bf Case 1}: In this case we have ${\ti d}^4(\ti x_0) \geq 2N\ti t_0$.
For $\ti y $ with  ${\ti d}^4(\ti y) \geq \frac N 2$, we obtain
\begin{align}
{\ti d}^4(\ti y) \geq \frac N 2 \geq N \ti t_0 \geq N \ti t
\label{(*)}
\end{align}
for all $ \ti t \leq \ti t_0$.
Hence, we see that
\begin{align}
\ti Q( \ti y, \ti t) 
\leq \frac N {  {\ti d}^4(\ti y)  }
 \leq 2  \label{mainest2}
\end{align}
for all $\ti t \leq \ti t_0$
in view of \eqref{(*)} and \eqref{mainest}.

We also have that ${\ti d}^4(\ti x_0) = N \geq \frac N 2$ and so the above
estimate also  holds for $\ti y = \ti x_0$ and $\ti t = \ti t_0$.
We calculate 
\begin{align}
 \frac N 2 
  &\leq {\ti d}^4(\ti y) = \left( \frac 1 c - |\ti y| \right)^4
\notag\\
 \Longleftrightarrow \quad  \left(\frac 1 c - |\ti y|\right)
  &\geq \frac{ N^{ \frac 1 4} }{2^{\frac 1 4} }
\notag\\
 \Longleftrightarrow \hspace{3mm}\qquad\quad\ \  |\ti y|
  &\leq -  \frac{ N^{\frac 1 4} }{ 2^{\frac 1 4} }  + \frac 1 c
\,.
\label{(**)}
\end{align}
Furthermore ${\ti d}^4  (\ti x_0) = N$ implies
$|\ti x_0| = -N^{\frac 1 4} + \frac 1 c$.
Assume that $\ti y $ is an arbitrary point with
$\ti  y \in B_{   N^{\frac 1 4}  / 400 }(\ti x_0)$.
Then we have 
\begin{align}
|\ti y| \leq | \ti x _0   |  + |  \ti x_0 - \ti y |
&= -N^{\frac 1 4} + \frac 1 c +  |  \ti x_0 - \ti y | \cr
&\leq  -N^{\frac 1 4} + \frac 1 c + \frac {N^{\frac 1 4}}{400} \cr
&\leq   \frac 1 c -  \frac {N^{\frac 1 4}}{ 2^{\frac 1 4}}  \label{didly}
\intertext{and hence, in view of \eqref{(**)}}
{\ti d}^4(\ti y) &\geq  \frac N  2\,.
\end{align}
Therefore $\ti  y \in B_{ N^{\frac 1 4}/ 400 }(\ti x_0)$ and $\ti t\leq \ti
t_0 \leq \frac 1 2$ implies in Case 1 that
\begin{align*}
 \ti Q( \ti y, \ti t)   \leq 2\,,\quad\text{and}\quad
 \ti Q( \ti x_0, \ti t_0) = 1\,,
\end{align*}
in view of \eqref{mainest2} and the definition of $\ti x_0$ and $\ti
t_0$.

\vspace*{5mm}
\noindent {\bf Case 2.}
In this case we have $1\geq \ti t_0 > \frac 1 2.$ 
For all $\ti t \leq \frac 1 2$ and $\ti y$ with ${\ti d}^4(\ti y) \geq \frac N 2$ we
have
\begin{align*}
{\ti d}^4(\ti y) \geq \frac N 2 \geq N \ti t
\end{align*}
and hence
\begin{align}
\ti Q( \ti y, \ti t) & \leq \frac N { {\ti d}^4(\ti y)} \leq 2 \label{zwisch1}
\end{align}
in view of \eqref{mainest}.
For $\ti t_0 \geq \ti t \geq \frac 1 2$ we have
\begin{align}
\ti Q( \ti y, \ti t) & \leq \frac {k_0}{ \ti t} \leq 2
k_0 \label{zwisch2}, 
\end{align}
in view of \eqref{assum1}.
Note that we may assume without loss of generality that $k_0 \geq 1$.
Now we know from \eqref{didly} that 
$y \in B_{ N^{\frac 1 4} /400}(\ti x_0)$ implies that   $\ti
  d^4(\ti y) \geq \frac N 2 $.
Hence, using the inequalities \eqref{zwisch1} and \eqref{zwisch2}, we see that 
\begin{equation}
 \ti Q( \ti y, \ti t)  \leq 2k_0\,, 
 \quad\text{and}\quad
 \ti Q( \ti x_0, \ti t_0) = 1\,,
\label{EQkeyestimate}
\end{equation}
for all $y \in B_{{N^{\frac 1 4}}/{400}}(\ti x_0)$ and $t \in  [0,\ti t_0]$.
We have shown that in both Case 1 and Case 2 we obtain the estimate \eqref{EQkeyestimate}.
Now we use Corollary \ref{CY2}  to obtain  a contradiction.

We use $v: B_R(0) \times [0,\ti t_0] \to \R$ to denote the rescaled solution $\ti u:
B_{{N^{\frac 1 4}}/{400}}(\ti x_0) \times [0,\ti t_0]
\to \R$.
That is $v(\cdot,\cdot) = \ti u(\cdot - \ti x_0,\cdot)$, $R = N^{\frac 1 4}/400$, $\ti t_0
\leq 1$. The definition of  $\ti u$ guarantees that $\ti u(\ti x_0,0)
=0$, and hence we have $v(0,0)=0$.
Also, using \eqref{assum2} and the fact that $c^4 \leq 1/N$ and $N >1$, we see that
\begin{align}
 \sup_{B_{1/c}(0)} \sum_{i=1}^{2n+2} |\grad^i v|^2(\cdot ,0) & =
 \sup_{B_1(x_0)} \sum_{i=1}^{2n+2} c^{2i} |\grad^i
 u|^2(\cdot,0) \cr
& \leq  \sup_{B_1(x_0)} \sum_{i=1}^{2n+2}  \frac{1}{N^{i/2}} |\grad^i
u(\cdot ,0)|^2\cr
& \leq \frac{ k_1}{N^{1/2}}.
\end{align} 
Hence, combining this estimate with the fact that $v(0,0)=0$ and by
choosing $N$ sufficiently large, we may
assume w.l.o.g. that  
\begin{align}
 \sup_{B_1(0)} \sum_{i=0}^{2n+2} |\grad^i v|^2(\ti x,0)  \leq \ti \ep(N) \label{newA2}
\end{align}
where $\ti \ep(N) \to 0$ as $N \to \infty$ ( $k_0,k_1,n$ are fixed in
this argument). Define  $2\rho = R = N^{\frac 1
  4}/400  \leq {1/c}$ ($c \leq \frac{1}{N^{1/4}}$ as we mentioned
above) and $p = p(n) = n+1$. Then $\rho \to \infty$ as $N \to \infty$.
Corollary \ref{CY2} implies that
\begin{align} 
\frac{d}{dt}E^{p}_\eta(v) + E^{p+1}_\eta(v)
 \le \frac{C}{\rho^{4p}}\int_{\R^n}
 |\Delta v|^2\gamma^{s-4p }\, \label{step1}
\end{align}
for all $t \leq \ti t_0$ for all $s> 4p+4$, where $\eta = \ga^s$, and
$\ga$ is a cutoff function as in \eqref{EQgamma}, and $C =
C(n,s)$. Choose $s = 4p(n) + 5 = 4n + 9$, so that $C = C(n)$.  
We know from the estimate \eqref{EQkeyestimate} that $Q (v) = |\lap
v|^2 \leq 2k_0$ on $ B_R(0) \times [0,\ti t_0] $
and
hence, combining this with \eqref{step1} we have
\begin{align} 
\frac{d}{dt}E^{p}_\eta(v) + E^{p+1}_\eta(v)
 & \leq \frac{C}{\rho^{4p}}\int_{\R^n}
 |\Delta v|^2\gamma^{s-4p} \cr
&\leq \frac{C}{\rho^{4p}} 
\int_{B_{2\rho}} 2k_0\cr
& = 2\omega_nk_0C \rho^{n-4p}  \label{remember}
\end{align}
which implies that 
\begin{align*} 
E^{p}_\eta(v)(t)  \leq (2\omega_n k_0 C + \omega_n k_1)  \rho^{n-4p} = (2\omega_n k_0 C + \omega_n k_1)  \rho^{-3n-4} 
\end{align*}
for all $t \leq \ti t_0 \leq 1$
in view of the fact that 
\begin{align*} 
E^{p}_\eta(v)(0) =  \int_{\R^n} |\Delta^{p} v_0|^2( \ti x  )\gamma^{s}d \ti x
& \leq \int_{B_{2\rho}(0)}  |\Delta^{p} v_0|^2(\ti x) d \ti x\cr 
& \leq \int_{B_{1/c}(0)}  |\Delta^{p} v_0|^2(\ti x) d \ti x\cr
& = \int_{B_{1/c} (0) }  c^{4p } |\Delta^{p} u_0|^2(c \ti x) d \ti x\cr
& = c^{4p - n } \int_{B_1{0}}   |\Delta^{p} u_0|^2( x) dx
  \cr
& \leq k_1 \omega_n c^{4p  - n }\cr
&\leq k_1 \omega_n \rho^{n-4p}  
\end{align*}
 where we have used assumption \eqref{assum2} again, the definition of
 $v$, the scaling properties of the derivatives of $u(cx,0)$,
 and the fact that $1/c \geq 2\rho \geq \rho $.
In particular,
\begin{align*} 
\int_{B_{1}(0)}  |\Delta^{p} v |^2 \leq (\omega_n 2k_0 C + k_1
\omega_n)  \rho^{-3n-4}\to 0 
\end{align*}
as $\rho \to \infty,$ in view of the fact that $ p = p(n) = n+1$.
We have shown  that
\begin{align*} 
\int_{B_{1}(0)}  |\Delta^{p} v |^2 \leq \ep_p(k_0,k_1,n,\rho)  
\end{align*}
for all $t\leq \ti t_0 \leq 1$ where $ \ep_p(k_0,k_1,n,\rho)    \to 0$ as
$\rho \to \infty$, that is, as $N \to \infty$.
We can similarly show that 
\begin{align*} 
\int_{B_{1}(0)}  |\Delta^{p-1} v |^2 \leq \ep_{p-1}(k_0,k_1,n,\rho) . 
\end{align*}

We also have
\begin{align*} 
\frac{d}{dt} \int_{B_1(0)} |\lap^{p-2} v|^2 & =  2 \int_{B_1(0)}
(\lap^{p-2} v)(\lap ^p v)\cr
& \leq  \int_{B_1(0)}
|\lap^{p-2} v|^2  +  \int_{B_1(0)} |\lap^p v|^2\cr
&\leq  
\int_{B_1(0)} |\lap^{p-2} v|^2 + \ep_p
\end{align*}
in view of Young's inequality, and the estimate just shown, 
and hence, after integrating in time from $0$ to $\ti t_0 \leq 1$
we see that 
\begin{align} 
\int_{B_1(0)} |\lap^{p-2} v|^2 \leq \ep_{p-2}(N) \label{inttime}
\end{align}
for all $t\in [0,\ti t_0 \leq 1] $ with $\ep_{p-2}(N) \to 0$ as $N \to \infty$: we leave out dependence on
$k_1,k_0,n$ since these  variables are fixed.
More explictly:  $f(t):= e^{-t} (\int_{B_1(0)} |\lap^{p-2} v|^2)(t) -
2t \ep_p $ satisfies $\frac{df}{dt}(t) \leq 0$ for all $0\leq t \leq
\ti t_0$
and $f(0) =  (\int_{B_1(0)}
|\lap^{p-2} v|^2)(0) \leq \ti \ep(N) \to 0 $ as $N \to \infty$ in
view of \eqref{newA2}, and so,  integrating $f$ from $0$ to $t_0$, we
see that the estimate \eqref{inttime} is true.


Continuing in this way, we get, for $N$ sufficiently large:
\begin{equation} 
\int_{B_1(0)} |\lap^{l} v|^2 \leq \ep_{l}(N) \label{lapsmall}
\end{equation}
for $l = p,p-2,p-4, \ldots, 1$ (or $0$).
Starting with $p-1$ instead of $p$, we similarly get
\begin{equation} 
\int_{B_1(0)} |\lap^{l} v|^2 \leq \ep_{l}(N)  \label{lapsmall2}
\end{equation}
for $l = p-1,p-3,p-5, \ldots, 1$ (or $0$), where $\ep_{l}(N) \to 0$ as $N \to
\infty$.
That is \begin{equation} 
\int_{B_1(0)} |\lap^{l} v|^2 \leq \ep_{l}(N) \label{lapsmall3}
\end{equation}
for $l \in \{0, \ldots, p\}$, where $\ep_{l}(N) \to 0$ as $N \to
\infty$.
 
Using the $L^2$ estimates, Lemma \ref{L2est} from the Appendix, and
the estimate \eqref{lapsmall3}  we see that
\begin{equation}
\int_{B_{1/2}(0)} |\grad^l v|^2 \leq  \hat \ep (N)
\end{equation}
for all $0\leq l \leq 2p= 2n +2$, where $\hat \ep(N) \to 0$ as $N \to
\infty$ (choose $\si = \si(n) = \frac{1}{4p(n)}$, so that $1- 2p \si
= 1/2$). 

Applying the Sobolev-Morrey inequality \cite[Theorem 6, Section 5.6.3]{Evans},  with
$p$, $k$ there equal to 2, and $2n+2$ respectively, we see that 
\[
\ti Q(0,t_0) = |\lap v|^2(0,t_0) \leq C(n) \Big(\sum_{l=0}^{2n+2}
\int_{B_{1/2}(0)} |\grad^l v|^2(\cdot,t_0)\Big)^{\frac 1 2} \leq
C(2n +3)^{\frac 1 2} (\hat \ep)^{\frac 1 2}
\]
and hence 
\[
\ti Q(0,t_0)  \to 0
\]
as $N \to \infty$. This contradicts the fact that  $\ti Q(0,t_0)  = 1$.

\end{proof}

The proof of Theorem \ref{main2} is essentially the same as the proof of Theorem \ref{main1}.

\begin{proof}[Proof of Theorem \ref{main2}.]
Replace $Q(x,t)$ by $Q_s(x,t)$, $e(x,t)$ by $e_s(x,t)$, $b(x,t)$
by $b_s(x,t)$, and $Q(u)$ by $Q_s(u)$ and repeat the above proof. At the point where
$|\lap v|^2 = Q (v) \leq 2k_0$ on $ B_R(0) \times [0,\ti t_0] $ is used in the
inequality \eqref{remember}, use instead the 
fact that 
$|\lap v|^2 \leq  |\grad^2 v|^2 \leq Q_s (v) \leq 2k_0$. 
Also choose $p(n)= n+(s/2)$ or $p= n + \frac{(s+1)}{2}$ in the proof: whichever is an integer.
The last part of the proof, where Morrey's embedding Theorem is used, has to be
slightly modified:
$Q_s (v)(0,t_0) = 1$ implies that $|\grad^r v|(0,t_0) \geq
\de(s)>0$ for some $r \in \{ 1, \ldots, s\}$ for some small $\de(s) >0$:
otherwise
the sum of the terms appearing in $Q_s(v)(0,t_0)$ would be less
than one.

Applying the Sobolev-Morrey inequality \cite[Theorem 6, Section 5.6.3]{Evans} with
$p,k$ there equal to 2, $2n+s$ respectively, we see that 
\begin{align*}
0 &< \de^2(s) \leq |\grad^r v|^2(0,t_0)
\\
  &\leq C(n,s) \Big( \sum_{l=0}^{2n+s} \int_{B_{1/2}(0)}|\grad^l v|^2(\cdot,t_0) \Big)^{1/2}
\\&\leq C (2n+1+s)^{\frac 1 2}  (\hat \ep (N))^{1/2},
\end{align*}
which leads to a contradiction if $N$ is chosen large enough, since $\hat
\ep(N) \to 0$ as $ N \to \infty$.
\end{proof}


\section{A priori Energy Estimates}
\label{STaprioriestimates}

In this section we shall prove some estimates for the weighted energies
\begin{equation}
E^k_\eta(u) = \int_{\R^n} |\Delta^ku|^2\,\eta,\quad \eta\in C^\infty_{loc}(\R^n),\quad \text{supp }\eta \subset\subset \R^n,
\label{EQenrgy}
\end{equation}
where $k\in\N_0$ and $u:\R^n\times[0,T] \rightarrow\R$ is a smooth solution to \eqref{bi}.
In the above equation and in that which follows all integrals are with respect to Lebesgue measure.
Note that $E^k_\eta$ are all finite for 
any $k\in\N_0$ and $t\in[0,T]$ since $u: \R^n \times [0,T] \to \R$ is smooth.
The purpose of the a priori estimates in this section is to quantify how global quantities such as the various
Sobolev norms of the solution behave along the flow \eqref{bi}.

\begin{lemma}
\label{LMlm1}
Let $u:\R^n\times[0,T]\rightarrow\R$ be a smooth solution to \eqref{bi}.
For all $t\in[0,T]$,
\begin{align}
\frac{d}{dt}&E^k_\eta(u) + 2E^{k+1}_\eta(u)
\notag\\
 &=
  - 2\int_{\R^n} (\Delta^{k+1}u)(\nabla_i\Delta^ku)(\nabla_i\eta)\,
  + 2\int_{\R^n} (\Delta^ku)(\nabla_i\Delta^{k+1}u)(\nabla_i\eta)\,.
\label{EQlm1}
\end{align}
for all $k \in \N_0$.
\end{lemma}
\begin{proof}
Differentiating,
\begin{align*}
\frac{d}{dt}E^k_\eta(u)
 &= 2\int_{\R^n} (\Delta^ku)(-\Delta^{k+2}u)\,\eta\,
\\
 &= 2\int_{\R^n} (\nabla_i\Delta^ku)(\nabla_i\Delta^{k+1}u)\,\eta\,
  + 2\int_{\R^n} (\Delta^ku)(\nabla_i\Delta^{k+1}u)(\nabla_i\eta)\,
\\
 &= -2\int_{\R^n} |\Delta^{k+1}u|^2\,\eta\,
  - 2\int_{\R^n}      (\nabla_i\Delta^ku)(\nabla_i\eta) (\Delta^{k+1}u)\,
\\&\qquad
  + 2\int_{\R^n} (\Delta^ku)(\nabla_i\Delta^{k+1}u)(\nabla_i\eta)\,
\\
 &= -2E^{k+1}_\eta(u)
\\&\qquad
  - 2\int_{\R^n} (\Delta^{k+1}u)(\nabla_i\Delta^ku)(\nabla_i\eta)\,
  + 2\int_{\R^n} (\Delta^ku)(\nabla_i\Delta^{k+1}u)(\nabla_i\eta)\,.
\end{align*}
Rearranging gives the lemma.
\end{proof}

We now specialise by setting $\eta = \gamma^s$, $s>0$ to be chosen, and $\gamma \in
C^\infty_{\text{loc}}(\R^n)$ satisfying
\begin{equation}
\label{EQgamma}
\tag{\ensuremath{\gamma}}
\chi_{B_\rho(0)} \le \gamma \le \chi_{B_{2\rho}(0)},\quad \rho>0,
\quad\text{and}\quad
|\nabla\gamma| \le \frac{c_{\gamma}}\rho,\quad|\nabla^2\gamma|\le \frac{c_{\gamma}}{\rho^2}\,,
\end{equation}
where $c_{\gamma}\ge1$ is an absolute constant depending only on $n$.

\newcommand{\cg}{c_{\gamma}}

In the following proofs we make extensive use of the elementary inequality
\begin{equation}
\label{EQcauchy}
ab \le \varepsilon a^2 + \frac1{4\varepsilon}b^2
\end{equation}
for $a$, $b$ real numbers and $\varepsilon>0$.

\begin{lemma}
\label{LMlm2}
Let $u\in C^\infty_{\text{loc}}(\R^n)$.
Suppose $\eta=\gamma^s$, $s>8$, and $\gamma$, $\cg$ are as in \eqref{EQgamma}.
Then for any $\varepsilon>0$ and for all $k \in \N$ we have
\begin{align*}
- 2\int_{\R^n} (\Delta^{k+1}u)(\nabla_i\Delta^ku)(\nabla_i\eta)\,
&\le \varepsilon E_\eta^{k+1}(u)
 + \frac{c_{1}(\varepsilon,s,n) }{\rho^{8}}\int_{\R^n} |\Delta^{k-1}u|^2\gamma^{s-8}\,,
\end{align*}
where $c_{1}(\varepsilon,s,n) < \infty $ is a constant depending on $
\varepsilon,s$ and $n$.
\end{lemma}
\begin{proof}
Throughout the proof $\delta_i$ denote positive parameters to be chosen.
Using \eqref{EQcauchy} and \eqref{EQgamma},
\begin{align}
- 2\int_{\R^n} (\Delta^{k+1}u)(\nabla_i\Delta^ku)(\nabla_i\eta)\,
&=
- 2s\int_{\R^n} (\Delta^{k+1}u)(\nabla_i\Delta^ku)(\nabla_i\gamma)\gamma^{s-1}
\notag\\
&\le \delta_1 E_\eta^{k+1}(u)
      + \frac{\cg^2s^2}{\delta_1\rho^2} \int_{\R^n} |\nabla\Delta^ku|^2\gamma^{s-2}\,.
\label{EQlm2begeq1}
\end{align}
Now,
\begin{align*}
\int_{\R^n}& |\nabla\Delta^ku|^2\gamma^{s-2}
\\
&=
-\int_{\R^n} (\Delta^ku)(\Delta^{k+1}u)\gamma^{s-2}
-(s-2)\int_{\R^n} (\Delta^ku)(\nabla_i\Delta^{k}u)(\nabla_i\gamma)\gamma^{s-3}
\\
&\le
\rho^2\delta_2\int_{\R^n} |\Delta^{k+1}u|^2\gamma^{s}
+ \frac1{4\delta_2\rho^2}\int_{\R^n} |\Delta^ku|^2\gamma^{s-4}
\\
&\qquad
+ \frac12 \int_{\R^n} |\nabla\Delta^ku|^2\gamma^{s-2}
+ \frac{\cg^2(s-2)^2}{2\rho^2}\int_{\R^n} |\Delta^ku|^2\gamma^{s-4}\\
&= \rho^2\delta_2\int_{\R^n} |\Delta^{k+1}u|^2\gamma^{s}
+ \frac12 \int_{\R^n} |\nabla\Delta^ku|^2\gamma^{s-2}\\
&\qquad 
\frac{1}{2\rho^2}(\frac{1}{2\de_2} +  \cg^2(s-2)^2)\int_{\R^n} |\Delta^ku|^2\gamma^{s-4}
\end{align*}
Absorbing the second term on the right into the left we obtain
\begin{align}
\notag
\int_{\R^n} &|\nabla\Delta^ku|^2\gamma^{s-2}
\\
&\le
2\rho^2\delta_2\int_{\R^n} |\Delta^{k+1}u|^2\gamma^{s}
+ \frac1{\rho^2}\Big(\frac1{2\delta_2} + \cg^2(s-2)^2\Big)
  \int_{\R^n} |\Delta^ku|^2\gamma^{s-4}\,.
\label{EQlm2mideq1}
\end{align}
We now need an interpolation inequality.
From Lemma \ref{LMinterp1} we know that
\begin{align*}
 \frac1{\rho^2}&\int_{\R^n} |\Delta^ku|^2\gamma^{s-4}
\\
&\leq \delta_3\rho^2\int_{\R^n} |\Delta^{k+1}u|^2\gamma^{s}
  + \frac{c_{\delta_3}}{\rho^6}
    \int_{\R^n} |\Delta^{k-1}u|^2\gamma^{s-8}
\,,
\end{align*}
where $c_{\delta_3} = c_{\delta_3}(\delta_3,s,n) = 2^4\big(\frac1{\delta_3} + 2^9s^4\cg^4\big)$.
Using this in \eqref{EQlm2mideq1} we obtain
\begin{align*}
\int_{\R^n} |\nabla\Delta^ku|^2\gamma^{s-2}
& \leq 2\rho^2 \de_2 \int_{\R^n} |\Delta^{k+1}u|^2\gamma^{s}
\\
& \hskip-1.5cm + \left(\frac{1}{2\de_2} + \cg^2(s-2)^2\right) \Big(
\delta_3\rho^2\int_{\R^n} |\Delta^{k+1}u|^2\gamma^{s}
  + \frac{c_{\delta_3}}{\rho^6}
    \int_{\R^n} |\Delta^{k-1}u|^2\gamma^{s-8}
 \Big)
\\
& =  \left( 2\rho^2 \de_2  + \de_3\rho^2 \left(\frac{1}{2\de_2} + \cg^2(s-2)^2\right)\right)
\int_{\R^n} |\Delta^{k+1}u|^2\gamma^{s}\\
&\quad + \frac{c_{\de_3}}{\rho^6} \left(\frac{1}{2\de_2} + \cg^2(s-2)^2\right)
\int_{\R^n} |\Delta^{k-1}u|^2\gamma^{s-8}
\end{align*}
and hence
\begin{align}
\frac1{\rho^2}\int_{\R^n} |\nabla\Delta^ku|^2\gamma^{s-2}
&\le
\bigg(
    \delta_3\Big(\frac1{2\delta_2} + \cg^2(s-2)^2\Big)
 + 2\delta_2\bigg)\int_{\R^n} |\Delta^{k+1}u|^2\gamma^{s}
\notag\\&\quad
 + \frac{c_{\delta_3}}{\rho^8}
                 \Big(\frac1{2\delta_2} + \cg^2(s-2)^2\Big)
   \int_{\R^n} |\Delta^{k-1}u|^2\gamma^{s-8}\,.
\label{EQstar}
\end{align}
Combining \eqref{EQstar} with \eqref{EQlm2begeq1} gives
\begin{align*}
- 2\int_{\R^n} &(\Delta^{k+1}u)(\nabla_i\Delta^ku)(\nabla_i\eta)\,
\\&\le
    \bigg[\delta_1 
      + \frac{\cg^2s^2}{\delta_1}
\bigg(
    \delta_3\Big(\frac1{2\delta_2} + \cg^2(s-2)^2\Big)
 + 2\delta_2\bigg)\bigg] E_\eta^{k+1}(u)
\\&\qquad
 + \frac{\cg^2s^2}{\delta_1}
   \frac{c_{\delta_3}}{\rho^8}
                 \Big(\frac1{2\delta_2} + \cg^2(s-2)^2\Big)
                                \int_{\R^n} |\Delta^{k-1}u|^2\gamma^{s-8}
\,.
\end{align*}
We choose $\delta_i = \delta_i(s,\ep,n) >0$ so that
\begin{equation*}
    \bigg[\delta_1 
      + \frac{\cg^2s^2}{\delta_1}
\bigg(
    \delta_3\Big(\frac1{2\delta_2} + \cg^2(s-2)^2\Big)
 + 2\delta_2\bigg)\bigg] \le \varepsilon
\,.
\end{equation*}
This can be achieved by the choice
\begin{align*}
\delta_1 = \frac{\varepsilon}4,\quad
\delta_2 = \frac{\varepsilon^2}{32\cg^2s^2},\quad
\delta_3 = \frac{\varepsilon^4}{8\cg^4s^2(16s^2 + \varepsilon^2(s-2)^2)}
\end{align*}
for example.
Recalling the definition of $c_{\delta_3} = c(\delta_3,s,n) = 2^4\big(\frac1{\delta_3} + 2^9s^4\cg^4\big)$ yields the result.
\end{proof}

\begin{lemma}
\label{LMlm3}
Let $u\in C^\infty_{\text{loc}}(\R^n)$.
Suppose $\eta=\gamma^s$, $s>8$, and $\gamma$, $\cg$ are as in \eqref{EQgamma}.
Then for any $\varepsilon>0$ and for all $ k \in \N$ we have
\begin{align*}
  2\int_{\R^n} (\Delta^ku)(\nabla_i\Delta^{k+1}u)(\nabla_i\eta)\,
&\le \varepsilon E_\eta^{k+1}(u)
 + \frac{c_{2}(\varepsilon,s,n)}{\rho^{8}}\int_{\R^n} |\Delta^{k-1}u|^2\gamma^{s-8}\,,
\end{align*}
where $c_2 (\varepsilon,s,n) < \infty $ is a constant depending on
$\varepsilon, s , n.$
%
\end{lemma}
\begin{proof}
Again, throughout the proof $\delta_i>0$ denote positive parameters to be chosen.
Integrating by parts,
\begin{align}
  2\int_{\R^n} (\Delta^ku)(\nabla_i\Delta^{k+1}u)(\nabla_i\eta)\,
 &= 
  - 2\int_{\R^n} (\nabla_i\Delta^ku)(\Delta^{k+1}u)(\nabla_i\eta)\,
\notag\\&\qquad
  - 2\int_{\R^n} (\Delta^ku)(\Delta^{k+1}u)(\Delta\eta)\,
\label{EQlm3begeq1}
\end{align}
Lemma \ref{LMlm2} deals with the first term:
\begin{equation}
  - 2\int_{\R^n} (\nabla_i\Delta^ku)(\Delta^{k+1}u)(\nabla_i\eta)\,
\le \frac{\varepsilon}{2} E_\eta^{k+1}(u)
 + \frac{c_{1} ( \frac{\varepsilon}{2},s,n)}{\rho^{8}}\int_{\R^n} |\Delta^{k-1}u|^2\gamma^{s-8}\,.
\label{EQlm3begeq2}
\end{equation}
Since $\Delta\eta = s\gamma^{s-1}\Delta\gamma + s(s-1)\gamma^{s-2}|\nabla\gamma|^2$ we have
\begin{align}
  - 2\int_{\R^n}& (\Delta^ku)(\Delta^{k+1}u)(\Delta\eta)\,
\notag\\
\label{EQstarstar}
&\le
  2\delta_1\int_{\R^n} |\Delta^{k+1}u|^2\eta\,
  + \frac{1}{\delta_1\rho^4}\Big(
       \cg^2s^2 + \cg^4s^2(s-1)^2
    \Big)\int_{\R^n} |\Delta^ku|^2\gamma^{s-4}\,\, ,
\end{align}
where we used the fact that $\ga^{s-2}(\cdot) \leq \ga^{s-4}(\cdot)$,
which is true since $0\leq \ga(\cdot) \leq 1$ on $\R^n$.
Lemma \ref{LMinterp1} yields the estimate
\begin{align*}
\frac{1}{\delta_1\rho^4}&\Big(
   \cg^2s^2 + \cg^4s^2(s-1)^2
\Big)\int_{\R^n} |\Delta^ku|^2\gamma^{s-4}
\\
&\le
\frac{\delta_2}{\delta_1}\Big(
   \cg^2s^2 + \cg^4s^2(s-1)^2
\Big)\int_{\R^n} |\Delta^{k+1}u|^2\gamma^{s}
\\&\quad+
\frac{c_{\delta_2}}{\delta_1\rho^8}\Big(
   \cg^2s^2 + \cg^4s^2(s-1)^2
\Big)\int_{\R^n} |\Delta^{k-1}u|^2\gamma^{s-8}
\end{align*}
where $c_{\delta_2} = c(\delta_2,s,n) =
2^4\big(\frac1{\delta_2} + 2^9s^4\cg^4\big)$.
Combining this with \eqref{EQstarstar} we get
\begin{align}
  - 2\int_{\R^n} (\Delta^ku)(\Delta^{k+1}u)(\Delta\eta)\,
&\le
  \bigg(2\delta_1 + \frac{\delta_2}{\delta_1}\Big(
       \cg^2s^2 + \cg^4s^2(s-1)^2
                \Big)
  \bigg)\int_{\R^n} |\Delta^{k+1}u|^2\eta\,
\notag
\\*&\hskip-1cm
 + \frac{c_{\delta_2}}{\delta_1\rho^{8}}
   \left(
       \cg^2s^2 + \cg^4s^2(s-1)^2
   \right)\int_{\R^n} |\Delta^{k-1}u|^2\gamma^{s-8}\,.
\label{EQlm3begeq3}
\end{align}
Combining \eqref{EQlm3begeq3} with
\eqref{EQlm3begeq1}-\eqref{EQlm3begeq2} and choosing $\delta_i =
\de_i(\varepsilon,n,s) >0$ so that
\[
\frac{\delta_2}{\delta_1}\Big(
       \cg^2s^2 + \cg^4s^2(s-1)^2
                \Big)
      + 2\delta_1
= \varepsilon/2
\]
yields the result.
One possible choice is
\begin{align*}
\delta_1 = \frac{\varepsilon}{16},\quad
\delta_2 = \frac{\varepsilon^2}{128\cg^2\Big( s^2 + \cg^2s^2(s-1)^2 \Big)}
\,.
\end{align*}
\end{proof}

\begin{corollary}
\label{CY1}
Let $u:\R^n\times[0,T]\rightarrow\R$ be a smooth solution to \eqref{bi}.
Suppose $\eta=\gamma^s$, $s>8$, and $\gamma$, $c_\gamma$, are as in
\eqref{EQgamma}, and $k \in \N$.
Then
\begin{equation*}
\frac{d}{dt}E^k_\eta(u) + \frac32E^{k+1}_\eta(u)
 \le \frac{c_3(n,s)}{\rho^{8}}\int_{\R^n} |\Delta^{k-1}u|^2\gamma^{s-8}dx\,,
\end{equation*}
where $c_3(n,s)$ is constant depending only on $n$ and $s$.
\end{corollary}
\begin{proof}
We combine Lemmata \ref{LMlm1}--\ref{LMlm3} as follows.
Lemma \ref{LMlm1} states that
\begin{align}
\frac{d}{dt}&E^k_\eta(u) + 2E^{k+1}_\eta(u)
\notag\\
\label{EQcy1eq1}
 &=
  - 2\int_{\R^n} (\Delta^{k+1}u)(\nabla_i\Delta^ku)(\nabla_i\eta)\,
  + 2\int_{\R^n} (\Delta^ku)(\nabla_i\Delta^{k+1}u)(\nabla_i\eta)\,.
\end{align}
The two terms on the right hand side are estimated by Lemma \ref{LMlm2} and Lemma \ref{LMlm3} respectively.
Adding together the estimates we find, for any $\varepsilon_1, \varepsilon_2 > 0$,
\begin{align*}
  - 2\int_{\R^n}& (\Delta^{k+1}u)(\nabla_i\Delta^ku)(\nabla_i\eta)
  + 2\int_{\R^n} (\Delta^ku)(\nabla_i\Delta^{k+1}u)(\nabla_i\eta)
\\
&\le (\varepsilon_1
      + \varepsilon_2) E_\eta^{k+1}(u)
 + \frac{c_1(\varepsilon_1,n,s) + c_{2}(\varepsilon_2,n,s) }{\rho^{8}}\int_{\R^n} |\Delta^{k-1}u|^2\gamma^{s-8}.
\end{align*}
In particular choosing $\varepsilon_i = \frac14$ and combining this with \eqref{EQcy1eq1} we have
\begin{align*}
\frac{d}{dt}E^k_\eta(u) + 2E^{k+1}_\eta(u)
\le \frac12
      E_\eta^{k+1}(u)
 + \frac{c_3}{2\rho^{8}}\int_{\R^n} |\Delta^{k-1}u|^2\gamma^{s-8}\,,
\end{align*}
where $c_3$ is a constant depending only on $s$ and $n$.
Absorbing the first term on the right into the left yields the claimed estimate.
\end{proof}

\begin{corollary}
\label{CY2}
Let $u:\R^n\times[0,T]\rightarrow\R$ be a smooth solution to \eqref{bi}.
Suppose $k\in\N $, $\eta=\gamma^s$, $s>4k$, where $\gamma$, $\cg$ are as in \eqref{EQgamma}.
Then 
\begin{equation*}
\frac{d}{dt}E^k_\eta(u) + E^{k+1}_\eta(u)
 \le \frac{c_4(n,s)}{\rho^{4k}}\int_{\R^n} |\Delta u|^2\gamma^{s-4k}\,,
\end{equation*}
where $c_4(n,s)$ is a constant depending only on $n$ and $s$.
\end{corollary}
\begin{proof}

We first consider the case where $k=1$.
Using Lemma \ref{LMlm1} and integration by parts we find
\begin{align}
\frac{d}{dt}&E^1_\eta(u) + 2E^2_\eta(u)
\notag\\
 &=
  - 2\int_{\R^n} (\Delta^2 u)(\nabla_i\Delta u)(\nabla_i\eta)\,
  + 2\int_{\R^n} (\Delta u)(\nabla_i\Delta^2 u)(\nabla_i\eta)\,
\notag\\
 &=
  - 4\int_{\R^n} (\Delta^2 u)(\nabla_i\Delta u)(\nabla_i\eta)\,
  - 2\int_{\R^n} (\Delta u)(\Delta^2 u)(\Delta\eta)\,.
\label{EQzero}
\end{align}
We claim that
\begin{equation}
\label{EQStar}
  - 4\int_{\R^n} (\Delta^2 u)(\nabla_i\Delta u)(\nabla_i\eta)
\le \varepsilon E^2_\eta(u) + \frac{c(\varepsilon,n,s)}{\rho^{4}}\int_{\R^n} |\Delta u|^2\gamma^{s-4}\,,
\end{equation}
and
\begin{equation}
\label{EQStarstar}
  - 2\int_{\R^n} (\Delta u)(\Delta^2 u)(\Delta\eta)
\le \varepsilon E^2_\eta(u) + \frac{c(\varepsilon,n,s)}{\rho^{4}}\int_{\R^n} |\Delta u|^2\gamma^{s-4}\,,
\end{equation}
hold.
Given the above estimates, we may conclude the required statement for the case $k=1$ as follows. Choosing $\varepsilon = \frac12$ in each of \eqref{EQStar}, \eqref{EQStarstar} and combining with \eqref{EQzero} we find
\[
\frac{d}{dt}E^1_\eta(u) + 2E^2_\eta(u)
\le E^2_\eta(u) + \frac{c(n,s)}{\rho^{4}}\int_{\R^n} |\Delta u|^2\gamma^{s-4}\,,
\]
whereupon subtraction of $E^2_\eta(u)$ from both sides yields the desired estimate.

The estimate \eqref{EQStarstar} is \eqref{EQstarstar} with $\delta_1 = \varepsilon/2$ and $k=1$.
It remains to prove the estimate \eqref{EQStar}.
We compute
\begin{equation}
\label{EQstar1}
  - 4\int_{\R^n} (\Delta^2 u)(\nabla_i\Delta u)(\nabla_i\eta)
\le \delta_1 E^2_\eta(u) + \frac{4\cg^2s^2}{\delta_1\rho^2}\int_{\R^n} |\nabla\Delta u|^2\gamma^{s-2}\,.
\end{equation}
Now estimate
\begin{align*}
\int_{\R^n}& |\nabla\Delta  u|^2\gamma^{s-2}
\\
&=
-\int_{\R^n} (\Delta  u)(\Delta^2u)\gamma^{s-2}
-(s-2)\int_{\R^n} (\Delta  u)(\nabla_i\Delta u)(\nabla_i\gamma)\gamma^{s-3}
\\
&\le \rho^2\delta_2\int_{\R^n} |\Delta^2u|^2\gamma^{s}
+ \frac12 \int_{\R^n} |\nabla\Delta  u|^2\gamma^{s-2}\\
&\qquad 
\frac{1}{2\rho^2}\Big(\frac{1}{2\de_2} +  \cg^2(s-2)^2\Big)\int_{\R^n} |\Delta  u|^2\gamma^{s-4}\,.
\end{align*}
Absorbing the second term on the right into the left we obtain
\begin{align}
\notag
\int_{\R^n} &|\nabla\Delta  u|^2\gamma^{s-2}
\\
&\le
2\rho^2\delta_2\int_{\R^n} |\Delta^2u|^2\gamma^{s}
+ \frac1{\rho^2}\Big(\frac1{2\delta_2} + \cg^2(s-2)^2\Big)
  \int_{\R^n} |\Delta  u|^2\gamma^{s-4}\,.
\label{EQstar2}
\end{align}
Combining \eqref{EQstar1} with \eqref{EQstar2} we find
\begin{align*}
  - 4&\int_{\R^n} (\Delta^2 u)(\nabla_i\Delta u)(\nabla_i\eta)
\\
&\le \Big(\delta_1 + 
2\rho^2\delta_2\frac{4\cg^2s^2}{\delta_1\rho^2}\Big)
    E^2_\eta(u)
 + \frac{4\cg^2s^2}{\delta_1\rho^2}\bigg(\frac1{\rho^2}\Big(\frac1{2\delta_2} + \cg^2(s-2)^2\Big)\bigg)  \int_{\R^n} |\Delta  u|^2\gamma^{s-4}\,.
\end{align*}
Choosing $\delta_1>0$ and $\delta_2>0$ such that $\Big(\delta_1 + 2\delta_2\frac{4\cg^2s^2}{\delta_1}\Big) \le \varepsilon$ yields \eqref{EQStar}.

Let us now continue by considering the case where $k\ge2$.
In this case we have $k-1 \in \N$, and $s >4k$ implies $s-4 >4k-4 = 4(k-1)$.
Using these facts and
Corollary \ref{LMinterp2}, we see that 
\begin{equation}
\frac{1}{\rho^{8}}\int_{\R^n} |\Delta^{k-1}u|^2\gamma^{s-8}
 \le \frac{1}{\rho^4} \int_{\R^n} |\Delta^{k}u|^2\gamma^{s-4}
   + \frac{c(s,n)}{\rho^{4k}}\int_{\R^n} |\Delta u|^2\gamma^{s-4k}
\label{EQcy2e1}
\,.
\end{equation}
and, using Corollary \ref{LMinterp2} again,
\begin{equation}
\frac{1}{\rho^4} \int_{\R^n} |\Delta^{k}u|^2\gamma^{s-4}
 \le \delta_1 \int_{\R^n} |\Delta^{k+1}u|^2\gamma^{s}
   + \frac{c(\delta_1,s,n)}{\rho^{4k}}\int_{\R^n} |\Delta u|^2\gamma^{s-4k}
\label{EQcy2e2}
\,.
\end{equation}
Combining \eqref{EQcy2e1} with \eqref{EQcy2e2} and then choosing 
$\delta_1 = \frac{1}{2c_3}$, where $c_3(n,s)$ is as in the previous Corollary, yields
\begin{equation}
\frac{c_3}{\rho^{8}}\int_{\R^n} |\Delta^{k-1}u|^2\gamma^{s-8}
 \le \frac12 \int_{\R^n} |\Delta^{k+1}u|^2\gamma^{s-4}
   + \frac{\ti c}{\rho^{4k}}\int_{\R^n} |\Delta u|^2\gamma^{s-4k}
\label{EQcy2e3}
\, ,
\end{equation}
for some $\ti c = \ti c(n,s)$. Using \eqref{EQcy2e3} to estimate the right hand side of Corollary \ref{CY1} we find
\begin{align*}
\frac{d}{dt}E^k_\eta(u) + \frac32E^{k+1}_\eta(u)
 &\le
      \frac{c_3}{\rho^{8}}\int_{\R^n} |\Delta^{k-1}u|^2\gamma^{s-8}
\\&\le
     \frac12 \int_{\R^n} |\Delta^{k+1}u|^2\gamma^{s-4}
   + \frac{\ti c}{\rho^{4k}}\int_{\R^n} |\Delta u|^2\gamma^{s-4k}
\,,
\intertext{which, after absorbing the first term on the right into the left, becomes}
\frac{d}{dt}E^k_\eta(u) &+ E^{k+1}_\eta(u)
\le
     \frac{c_4}{\rho^{4k}}\int_{\R^n} |\Delta u|^2\gamma^{s-4k}
\end{align*}
as required.

\end{proof}


\section{Uniqueness}
\label{STuniqueness}

In this section we prove that smooth solutions to \eqref{bi} which satisfy
$|\lap u|^2(\cdot,t) \leq \frac{k_0}{t}$ are uniquely determined by their
initial values. 
\begin{theorem}\label{uthm}
Let $v:\R^n \times [0,T] \to \R$, $T < \infty$  be a smooth solution to \eqref{bi}
which satisfies
\begin{align} 
 | \lap v |^2(x,t)  \leq \frac{k_0}{t} \label{introA1u}
\end{align} 
for some $k_0 \in \R$, for all $t \in [0,T]$ and all $x \in \R^n$ and
\begin{align} 
v_0 \equiv 0.
\end{align}
Then $ v \equiv 0$.
\end{theorem}

\begin{proof}
Since  $$ \sup_{B_1(0)} \sum_{i=0}^p |\grad^i v_0|^2 = 0 $$ for any
$p>0$, (c.f. (A2)), Theorem \ref{main1} tells us that $|\lap v|^2 (0,t) \leq  2N(n,k_0)$ for some
$N = N(n,k_0) \in \R$ for all $t \leq 1/ N$.
By setting $\ti v(\cdot,t) = v( \cdot -x_0,t)$ and using Theorem
\ref{main1} for $\ti v$, we see that $|\lap v|^2 (x_0,t) \leq N(n,k_0)$ 
for all $t \leq 1/ N$, for all $x_0 \in \R^n$.
Corollary \ref{CY2} implies that 
\begin{align}
 \partt E^p_{\eta}(v) + E^{p+1}_{\eta} (v) \leq 2CN
 \omega_n\rho^{n-4p} = 2CN
 \omega_n\rho^{-3n-4} 
\end{align}
where $p$ is now fixed and chosen to be $p(n) = n + 1$, and $\eta $ is
a non-negative cutoff function with $\eta = 1$ on $B_{\rho}(0)$, and
$C = C(n)$. To see this
repeat the argument from the inequality \eqref{step1} up to \eqref{remember}
but use this $v$ (instead of the $v$ appearing there) and use the fact that
$k_0 = 0$, and $|\lap v|^2 \leq N$ for all $ t\leq 1/N$ for this $v$. 
This implies that $E^p_{\eta}(v)(t) \leq 2CN\omega_n\rho^{-3n-4}$ for all $t
\leq 1/N \leq 1$, since
$E^p_{\eta}(v) $ is non-negative, and $E^p_{\eta}(v)( 0) = 0$.
Letting $\rho \to \infty$, we see that $\int_{\R^n} |\lap ^p v|^2 =0$ for all $t
\leq 1/N$.
Similarly $\int_{\R^n} |\lap ^{p-1} v|^2 = 0$ for all $t
\leq 1/N$.
Now use 
 \begin{align*}
\partt  \int_{B_1(0)} |\lap^{p-2} v|^2 & = 
2 \int_{B_1(0)} \lap^{p} v \lap^{p-2} v \cr
& \leq \int_{B_1(0)} |\lap^{p} v|^2 +  |\lap^{p-2} v|^2 
  = \int_{B_1(0)} |\lap^{p-2} v|^2 
\end{align*}
which tells us, after integrating, that $ \int_{B_1(0)}
|\lap^{p-2} v|^2 = 0 $  for all $t \leq 1/N$. Differentiating
$\int_{B_1(0)} |\lap^{p-3} v|^2$ w.r.t. time
and using  $\int_{B_1(0)}| \lap^{p-1} v|^2 = 0$ we obtain, using the same
argument, that  $ \int_{B_1(0)}
|\lap^{p-3} v|^2 = 0 $  for all $t \leq 1/N$. Continuing in this way,
we find that $\int_{B_1(0)}
|\lap^{l} v|^2 = 0$ for all $0 \leq l \leq p$, for all $t \leq 1/N$.
Similarly, we obtain $\int_{B_1(x_0)}
|\lap^{l} v|^2 = 0$ for all $0\leq l \leq p$, for all $t \leq 1/N$ for
all $x_0 \in\R^n$. In
particular, by choosing $l =0$, we see that $ v(\cdot,t) =0$ for all $t \leq
1/N$, $t \leq T$.
Repeating this argument for the function $\ti v(\cdot,\ti t) =
v(\cdot,\ti t + 1/N)$, we see that $v(\cdot,t) =0$ for all $t \leq T$, as required.
\end{proof}                                                                                                                                                                                                                             
\begin{corollary}\label{uniqueness}
Let $u, v:\R^n \times [0,T] \to \R$, $T < \infty$  be smooth solutions to \eqref{bi}
which satisfy
\begin{align} 
 |\lap v|^2(x,t) + |\lap u|^2(x,t) \leq \frac{k_0}{ t} \label{uniqueA1}
\end{align} 
for all $t \in [0,T]$ and all $x \in \R^n$ and
\begin{align} 
u_0(\cdot) = v_0(\cdot).
\end{align}
Then  $u \equiv v$.
\end{corollary}

\begin{proof} 
Set $s = u-v$. Then $s$ satisfies the conditions of Theorem \ref{uthm}. Hence
$s \equiv 0$ as required.
\end{proof}


\section{A Tychonoff-type solution and Non-uniqueness}\label{Tych}

In this section we describe a simple modification to the classical Tychonoff
counterexample, see \cite{Tych}, which establishes non-uniquness for complete solutions of the
polyharmonic heat equation.
We follow the construction given in \cite[Chapter 7, Section 1 (a), pp 211--213]{FritzJohn}.

Let $k\in\N$ and consider a solution $u:\R^n\times[0,T]\rightarrow\R$ to
\begin{align}
(\partial_t-\Delta^k)u &= 0\quad\qquad\text{on}\quad\R^n\times[0,T]\,,
\label{EQtik1}
\\
u(\cdot,0) &= u_0(\cdot)\,\ \quad\text{on}\quad\R^n\,.
\label{EQtik2}
\end{align}
We shall construct infinitely many solutions to
\eqref{EQtik1}-\eqref{EQtik2} which have zero as their initial
data.

For functions $g_j:[0,T]\rightarrow\R$ to be chosen, set
\[
u(x,t) = \sum_{j=0}^\infty g_j(t)x_1^{2jk}\,.
\]
The convergence of this series will be guaranteed by our choice of $g_j$, and verified later.
Differentiating formally, we find
\begin{align*}
\sum_{j=0}^\infty (\partial_tg_j)(t)x_1^{2jk}
 &= \partial_tu(x,t)
  = \Delta^ku(x,t)
\\
 &= \sum_{j=1}^\infty (2jk)(2jk-1)\cdots(2jk-2k+1)g_j(t)x_1^{2jk-2k}
\\
 &= \sum_{j=0}^\infty (2jk+2k)(2jk+2k-1)\cdots(2jk+1)g_{j+1}(t)x_1^{2jk}\,
\end{align*}
for all $j \in \N_0$.
We are thus led to the recurrence relation
\begin{equation}
\label{EQtik3}
(\partial_tg_j) = (2jk+2k)(2jk+2k-1)\cdots(2jk+1)g_{j+1}\,
\end{equation}
for all $j \in \N_0$.
We set $g_j(t) = \lambda(j,k)g_0^{(j)}(t)$, where $(g_0)^j$ refers to
$j$ temporal derivatives of $g_0$, and  $\lambda(j,k)$ is a constant to be determined depending only on
$j,k$. Using this choice of $g_j$, we see that \eqref{EQtik3} is
satisfied, provided that
\begin{equation*}
g_0^{(j+1)}\lambda(j,k) = (2jk+2k)(2jk+2k-1)\cdots(2jk+1)\lambda(j+1,k)g_0^{(j+1)}\,,
\end{equation*}
which for $g_0^{(j+1)}\ne0$ is equivalent to
\begin{equation*}
\frac{\lambda(j,k)}{\lambda(j+1,k)} = (2jk+2k)(2jk+2k-1)\cdots(2jk+1)\,.
\end{equation*}
Using $\lambda(0,k) = 1$, we see that this implies that $\lambda(j,k)
= \frac{1}{(2jk)!}$ for all $j \in \N_0$ (we use $0!:= 1$) .
Let us now set
\begin{equation*}
g_0(t) = \exp(-t^{-p})
\end{equation*}
for $t>0$ and $p>1$.

\begin{lemma}
\label{LMtikhonovestimate}
There is an absolute constant $\ep_0  >0$ and a $p >1$ such that the estimate
\[
\Big|g_0^{(j)}(t)\Big| \le \frac{j!2^j}{t^j}\exp\big(-\ep_0(2t)^{-p}\big)
\]
holds for all $t>0$.
\end{lemma}
\begin{proof}
Consider the function $h(z)=\exp(-z^{-p})$ for $p>1$.
Since $z^p := \exp(p\Log(z))$ is analytic on $\C\backslash(-\infty,0]$,
$h$ is analytic on $\C\backslash(-\infty,0]$: for polar coordinates $z
= re^{i\theta}$ with $\theta \in (-\pi, \pi)$ we are using $\Log(z):=
\log(r) + i \theta$, and hence $z^p = r^pe^{ip\theta}$.
For $0<\rho<t$, Cauchy's integral formula on $S_\rho(t+0i) = S_\rho(t)$, the circle in $\C$ with radius $\rho$
centred at $t+0i$, gives
\[
g_0^{(j)}(t)
 = h^{(j)}(t+0i)
 = \frac{j!}{2\pi i} \int_{S_\rho(t)} \frac{h(z)}{(z-t)^{j+1}}dz\,.
\]
This gives the estimate
\begin{equation}
\label{EQtik4}
|g_0^{(j)}(t)|
 \le \frac{j!}{\rho^j}\sup_{z\in S_\rho(t)} |h(z)|\,.
\end{equation}
In polar coordinates $z = r\exp(i\theta)$, $\theta \in (-\pi,\pi)$ we have
\[
h(z) = \exp\big(-r^{-p}e^{-ip\theta}\big)
     = \exp\big(-r^{-p}(\cos(p\theta) - i\sin(p\theta))\big)\,.
\]
Therefore
\begin{equation}
\label{EQrev0}
|h(z)| \le \exp(-r^{-p}\cos p\theta)\,.
\end{equation}
For $z\in S_\rho(t)$, we have
\[
-\frac{\pi}{2} < -\theta_0
 \le \theta
 \le \tan^{-1}( \rho/\sqrt{t^2-\rho^2} )
 =: \theta_0 < \frac{\pi}{2}
\]
Note that $\theta_0$ doesn't depend on $p$.
So we may choose $p>1$ such that $p\theta_0 < \frac{\pi}{2}$: this is
possible since $\theta_0 < \frac{\pi}{2}$.
We then have $\cos (p\theta) \ge \cos(p\theta_0) =: \ep_0 > 0$ for all
$\theta \in (-\theta_0, \theta_0)$.
Since $r\le 2t$ we may estimate
\[
-r^{-p} \cos p\theta \le -\ep_0 (2t)^{-p}
\]
for all $\theta \in (-\theta_0, \theta_0)$, 
which combined with our earlier estimate \eqref{EQrev0} yields
\[
\sup_{z\in S_\rho(t)} |h(z)|
  \le \exp\big(-\ep_0(2t)^{-p}\big)\,.
\]
Inserting this into the estimate \eqref{EQtik4} and choosing  $\rho = t/2$ finishes the proof.
\end{proof}

Lemma \ref{LMtikhonovestimate} implies
\begin{align*}
|u(x,t)|
 &\leq \sum_{j=0}^\infty |g_j(t)||x_1|^{2jk}
\\
 &=  \sum_{j=0}^\infty \frac{|g_0^{(j)}(t)|}{(2jk)!} |x_1|^{2jk}
\\
 &\le \sum_{j=0}^\infty \frac{j!2^j}{t^j(2jk)!}|x_1|^{2jk}\exp\big(-\ep_0(2t)^{-p})
\\
 &\le \exp\big(-\ep_0(2t)^{-p}\big)
      \sum_{j=0}^\infty \frac{j!}{(2jk)!}\Big(\frac{|x_1|^{2k}}{t/2}\Big)^{j} 
\\
&\le \exp\big(-\ep_0(2t)^{-p} \big)
      \sum_{j=0}^\infty \frac{1}{j!}\Big(\frac{|x_1|^{2k}}{t/2}\Big)^{j}
\\
 &= \exp\Big(-\frac{\ep_0}{(2t)^p} + \frac{|x_1|^{2k}}{t/2}\Big)\,.
\end{align*}
Here we have used $\frac{j!}{(2j)!} \leq \frac{1}{j!}$ for all $j \in
\N_0$ which may be
seen using induction.
Therefore $u$ is well-defined for every $t>0$.  Moreover, $p>1$ implies that the first term above always
dominates for small $t$ and so $u$ converges uniformly to zero on compact subsets of $\R^n$ as $t\searrow0$.
More precisely, let $K$ be a compact subset of $\R^n$ with diameter
$d$ and $0 \in K$. Then $|x| \le d$ and for $x\in K$
\[
\lim_{t\searrow0} |u|(x,t)
\le
  \lim_{t\searrow0}  \exp\Big(-\frac{\ep_0}{(2t)^p} + \frac{d^{2k}}{t/2}\Big)
= 0\,.
\]
A similar argument shows that all derivatives of $u$ exist and  converge uniformly to
zero on compact subsets of $\R^n$ as $t\searrow0$.
We explain this in the following.
Assuming $x = x_1$ satisfies $|x| \leq d$  where $d \geq 1$ and taking $s$ spatial derivatives formally we find
\begin{align*}
\left|(\partial_{x})^su(x,t)\right| & = \left| \sum_{j \geq s/(2k)}
(g_0)^j(t)(x)^{2jk-s} \frac{(2jk)(2jk-1) \ldots (2jk -s +1)} {(2jk)!} \right|
\\
& \leq \sum_{j \geq s/(2k)}
\left|(g_0)^j(t)(x)^{2jk-s} \right| \left| \frac{(2jk)(2jk-1) \ldots (2jk -s +1)} {(2jk)!} \right|
\\
& \leq \sum_{j \geq s/(2k)} |g_0^j|(t)|d|^{2jk-s} \frac{ 1}{(2jk-s)!} \\
& \leq \sum_{j \geq s/(2k)} |g_0^j|(t)(|d|^{2k} )^j  \frac{ 1}{(2jk-s)!} \\
& \leq  \exp(  -\ep_0(2t)^{-p})  \sum_{j \geq s/(2k)}  (
2d^{2k}/t)^j\frac{j!} {(2jk-s)!} \\
\displaybreak[4]
& \leq \exp(  -\ep_0(2t)^{-p})  \sum_{   \{ j \ | \ 2kj \geq s \}   }  (
2d^{2k}/t)^j\frac{(kj)!} {(2jk-s)!} \\
&\leq s! \exp(  -\ep_0(2t)^{-p})  \sum_{ \{ j \ | \ 2kj \geq s \} } \frac{(
  2d^{2k}/t)^j}{kj!}\\
& \leq s! \exp(  -\ep_0(2t)^{-p})  \sum_{ \{ j \ | \ 2kj \geq s \} } \frac{(
  2d^{2k}/t)^j}{j!}\\
&\leq s! \exp(  -\ep_0(2t)^{-p})  \exp( 2d^{2k}/t),
\end{align*}
which goes to zero as $t \downto 0$.
Here we used that $ \frac{(r!)^2}{(2r-s)!} \leq s!$ for all $r \geq
s$,$r,s \in \N_0$, which may be verified using induction on $r$. 
Since $s$ time derivatives of $u$ are formally given by $2ks$ spatial
derivatives of $u$, we see that all mixed derivatives (space and time) of $u$
exist for $t>0$ and converge uniformly on (spatial) compact sets $K \subset
\R^n$ to $0$.

By extending $u$ to be zero for all $t\le0$ we have a solution $u\in
C^\infty( \R^n \times(-\infty,\infty))$ to \eqref{EQtik1}-\eqref{EQtik2} which is
non-zero for $t>0$ and satisfies $u \equiv 0$ for all $t\le 0$.


\section{An example}\label{STexample}
Let $u_0 :\R^n \to \R$ be given by $u_0(x_1,x_2, \ldots,x_n):= 1$ if $x_1
>0$,  $u_0(x_1,x_2, \ldots,x_n):= 0$ if $x_1 \leq0$.
Setting $u(x,t):= \int_{\R^n} u_0(x-y) b(y,t) dy$,
with $b: \R^n \times (0,\infty) \to \R$  the bi-harmonic heat
kernel on $\R^n $, we see that the function $u: \R^n \times
(0,T) \to \R$ is smooth and solves 
$ \partt u(x,t) = -\lap^2u(x,t)$ for all $ t>0$ for all $x \in \R^n$,
and that $u(\cdot,t) \to u_0(\cdot)$ uniformly on any compact set  $K$
contained in $\R^n \backslash \{ x \in \R^n | x_1 = 0\}$. 
Furthermore, there  exists a $ k_0
 >0$ such that 
for all $s>0$ there exists a $x_s \in \R^n$ such that
$|\lap u|^2( x_s,s) = \frac{k_0}{s}$. 
The biharmonic heat kernel $b$ is given by 
$$b(x,t) = (2\pi)^{-n/2}
t^{-n/4}\int_{\R^n}e^{   i\langle w,x\rangle t^{-1/4} -|w|^4} dw.
$$

We verify of all these facts below.

We have (see the Appendix of  \cite{KL}, and the papers
\cite{GP}, \cite{GG1},\cite{GG2} for further, related and similar estimates) 
\begin{align}
 \left|\left(\partt\right)^l \left(\grad^k\right) b(x,t)\right| \leq C(k,l,m)  (t^{-p(k,l) + m/4} +
 t^{ (m-1)/4}) |x|^{-m}  \label{HKest}
\end{align}
for all $l,k \in \N_0, m \in \N_0$, for some $p(k,l) \in \N$  for all $x
\in \R^n$ for all $t >0 $.
This can be seen as follows.
Let $f: \R^n \to \R$ be $f(y):=  (2 \pi)^{-n/2} \int_{\R^n}e^{   -i \langle
  w, y\rangle-|w|^4} dw,$ 
so that $b(x,t) = t^{-n/4}f( -\frac{x}{t^{1/4}})$. Then $f$ is the Fourier transform of the function $l: \R^n \to \R$,
$l(w):= e^{-|w|^4}$ which is in $\curlS$, the so called Schwartz Space (see \cite{SW} Chapter
I.3 where this set of functions is defined and called the space of {\it testing
  functions}). Hence $f$ itself is in $\curlS$ (see \cite{SW}, Theorem 3.2
Chapter I.3), in particular $|\grad_{\al} f|(x) \leq \frac{c(|\al|,m) }{|x|^m }
$ for any $m \in \N_0$ and any multi-index $\al = (\al_1, \ldots,
\al_n)$: $\al_i \in \N_0 $ for all $i = 1, \ldots, n$, $| \al | := \al_1 + \al_2 + \ldots \al_n$, and
we have used the notation $\grad_{\al} f := \grad_{\al_1} \grad_{\al_2} \ldots
\grad_{\al_n} f.$

Using the represenation $b(x,t) =  t^{-n/4}f( -\frac{x}{t^{1/4}})$ and the
fact that $f$ is in $\curlS$  we get
 \begin{align*}
\left|\left(\partt\right)^l \left(\grad^k\right) b(x,t)\right|  
&\leq (t^{-p(k,l)} + t^{-1/4}) 
\sum_{0 \leq |\al|  \leq k+l} |\grad_{\al} f|\left(-\frac{x}{t^{1/4}}\right) \cr
& \leq (t^{-p(k,l)} + t^{-1/4}) \frac{c(k,l,m) }{ |x/t^{1/4}|^m} \cr
& = c(k,l,m) \frac{t^{-p(k,l) + m/4} +t^{(m-1)/4 }   }{ |x|^m} 
\end{align*}
which proves the estimate \eqref{HKest} since $m \in \N_0$ was arbitrary.
This shows that the function $u(x,t):= \int_{\R^n} u_0(x-y) b(y,t) dy =\int_{\R^n} u_0(z) b(x-z,t) dz
$ is well defined for any measurable  $L^{\infty}$ function $u_0: \R^n
\to \R$ for all $t>0$, and is differentiable in time and space for all
$t>0$ for all $x \in \R^n$ and the derivative is given by
differentiating under the integral sign (in view of the Lebesgue
dominated convergence Theorem):
\begin{align*}
\left(\partt\right)^l \left(\grad^k\right) u(x,t)  =\int_{\R^n} u_0(z) \left(\left(\partt\right)^l \left(\grad^k\right) b\right)(x-z,t) dz.
\end{align*}
Using the fact that $\partt b = -\lap^2 b$ (see below for an explanation),  we get
$u: \R^n \times (0,\infty) \to \R$ is smooth and satisfies
{$ \partt u = -\lap^2 u$.
Notice also that $\int_{\R^n} b(x,t) dx =  \int_{\R^n}  t^{-n/4} f(
-\frac{x}{t^{1/4}}) dx = \int_{\R^n} f(-z) dz = \int_{\R^n} b(z,1)
dz = 1$ (the last equality is explained below).
Hence, for $x \in B_{\ep}(z)$ where $ z = (z_1, \ldots, z_n)$ has $z_1
> 2\ep$, we have 
\begin{align}
 |u(x,t) - 1| & = \bigg|\int_{\R^n} b(x-y,t)(u_0(y) - 1) dy\bigg|
 \cr 
& =  \bigg| \int_{B_{\ep}(x)} b(x-y,t) (u_0(y) - 1) dy
\\&\quad +  \int_{\R^n \backslash B_{\ep}(x)} b(x-y,t) (u_0(y) - 1)  dy \bigg|\cr 
& = 0 + \bigg|\int_{\R^n \backslash B_{\ep}(x)} b(x-y,t) (u_0(y) - 1)  dy \bigg|\cr 
&  \leq \int_{\R^n \backslash B_{\ep}(x)} 2|b(x-y,t)| dy \cr 
&\leq \int_{\R^n \backslash B_{\ep}(x) } \frac{c(m,n)t^{m-n/4}}{|x-y|^{4m}} dy \cr 
&\leq  C(\ep,m,n)t^{m-n/4} \cr
&\leq C(\ep,m,n) t^2
\end{align}
for $m > 2 + n/4$ for $t\leq 1$ which goes to zero as $t \to 0$.
Similarly $|u(x,t)| \leq C(\ep,m,n)t^2$ goes to zero for all $x \in B_{\ep}(z)$ 
where $ z = (z_1, \ldots, z_n)$ has $z_1
< -2\ep$. Hence $u(\cdot,t) \to u_0$ uniformly on compact sets $K
\subseteq \R^n \backslash \{ x \in \R^n \ | \ x_1 = 0\}$.
The definition of $u_0$ and $u$ guarantees that $u(cx,c^4t) = u(x,t)$
for all $c,t>0$. We verify this now. Notice first that 
\begin{align}
b(cx,c^4t) & = (2\pi)^{-n/2}
(c^4t)^{-n/4}\int_{\R^n}e^{   i(c^4t)^{-1/4} \langle w,cx\rangle-|w|^4} dw\cr
& =c^{-n} (2\pi)^{-n/2}\int_{\R^n}e^{   i t^{-1/4} \langle
  w,x\rangle-|w|^4} dw
\end{align}
that is $b(cx,c^4t) = c^{-n}b(x,t) $ for all $ x \in \R^n$ for all $c>0$.
Also, the definition of $u_0$ guarantees that $u_0(cz) = u_0(z)$ for
all $z \in \R^n$ and all $c>0$.
Making a change of variable $y =  cw$ in the definition of $u$, and
then using $b(cw,c^4t) = c^{-n}b(w,t) $ and the property of $u_0$ just
mentioned, we calculate
\begin{align}
u(cx,c^4t):= \int_{\R^n} u_0(cx-y) b(y,c^4t) dy 
&= \int_{\R^n} u_0(cx-cw) b(cw,c^4t) c^n dw \cr
& = \int_{\R^n} u_0(x-w) b(w,t) dw \cr
& = u(x,t).
\end{align}
There must exist a point $(x_0,t_0) \in \R^n \times \R^+$ with
$\lap u(x_0,t_0) \neq 0$: if not, then $\partt u = -\lap^2 u = 0 $ for
all $t >0$, and hence $u(x,t) = u(x,s)$ for all $0<s<t$, and hence ,
using $u \to u_0$ on $\R^n \backslash \{ x \in \R^n | x_1 = 0\}$ as $t
\downto 0$
as explained above, we have $u(x,t) = 1$ on  $\R^n \cap \{ x \in
\R^n | x_1 > 0\}$, $u(x,t) = 0$ on  $\R^n \cap \{ x \in \R^n | x_1 <
0\}$ for all $t>0$, which contradicts the fact that $u(\cdot,t): \R^n
\to \R$ is smooth for $t>0$.

So there exists $(x_0,t_0) \in \R^n \times \R^+$ with $|\lap u(x_0,t_0)|^2 \neq
0$. Now $u(cx,c^4,t) = u(x,t)$ for all $t>0$, for all $x \in \R^n$ implies that
(take the Laplacian w.r.t $x$ of both sides) $c^2 (\lap u)(cx,c^4t) =  (\lap
u)(x,t)$ which implies $|\lap u|^2(cx,c^4t) = \frac{1}{c^4} |\lap u|^2(x,t)$.
In particular, choosing $t = t_0$, $c^4 = (s/t_0)$ and $x = x_0$ we find
\[
|\lap u|^2((s/t_0)^{1/4} x_0, s) = \frac{t_0}{s} |\lap
  u|^2(x_0,t_0)
\]
and hence 
$|\lap u|^2( x_s,s) = \frac{k_0}{s}$ where $k_0 = t_0 |\lap
  u(x_0,t_0)|^2 \neq 0$ and $x_s = (s/t_0)^{1/4} x_0$.
Using an almost identical argument, we see that for all $t>0$, there
must be points $y(t) \in \R^n$ such that 
$(\lap^2 u)( y(t),t) = \frac{k_1}{t}$ for some fixed $k_1 \in \R$, $k_1
\neq0 $.

The fact that $\partt b = -\lap^2 b$ can be seen as follows. Using
Theorem 1.7 of Chapter I.1 in \cite{SW}, we have
$-|x|^4  e^{-|x|^4 t} = \partt (e^{-|x|^4 t}) =(\partt \widehat b)(x,t) =
\widehat{(\partt b)}(x,t) = \widehat{( -\lap^2 b)}(x,t)$, and hence, taking 
the inverse of the Fourier transform, we get $(\partt b) = -\lap^2 b$ (note that $(\partt \widehat b)(x,t) =
\widehat{(\partt b)}(x,t) $ is true in view of the Lebesgue dominated
convergence Theorem and the estimates \eqref{HKest} and the inverse of
the Fourier transform exists in view of Corollary I.21 in I.1 of
\cite{SW} and the fact that $b$ is in $\curlS$).
The fact that $\int_{\R^n} b(z,1)
dz = 1$
may be seen by looking at how $b$ was derived: 
Let $ u_0: \R^n \to \R$ be a smooth function which is equal to $1$ on
$B_1(0)$ and has  compact support on $B_2(0)$.
Hence $u_0$ is in $\curlS$, and the Fourier transform
$\widehat{u_0} $ of $u_0$ is also in $\curlS$.
We only take the Fourier transform in the space direction in that
which follows.
Write $u(x,t) = (b(\cdot,t) * u_0)(x)$ so $\partt u = -\lap^2 u$ as
explained above, and
\begin{align}
\widehat u (x,t)  & = \widehat b (x,t) \widehat{u_0}(x)   
= e^{-t|x|^4} \widehat u_0(x) 
\end{align}
(see Theorem 1.4 in I.1 os \cite{SW}) and hence $\widehat u(\cdot,t) \to  \widehat{u_0}(\cdot)$
in the $L^2$ sense as $t \downto 0$.
But this implies $u(\cdot,t) \to u_0(\cdot)$ in the $L^2$ sense as $t \downto 0$,
in view of the fact that the
$L^2$ norm is preserved for the Fourier transform (and the inverse
Fourier transform) of a function in
$L^2\cap L^1$ (see Theorem 2.1 and Theorem 2.4 in Chapter I.2 of \cite{SW}).
In particular this shows $\int_{\R^n} b =1$:
If $1 \neq c_0 := \int_{\R^n} b \neq 0$, then  for $x \in B_{1/2}(0)$
we have 
\begin{align}
 |u(x,t) - 1/c_0| & = 
\bigg|\int_{\R^n} b(x-y,t)(u_0(y) - 1) dy\bigg|
 \cr 
& = \bigg|\int_{\R^n \backslash B_1(0)} b(x-y,t)(u_0(y) - 1) dy\bigg|
 \cr 
& = \bigg|\int_{\R^n \backslash B_1(0)} b(x-y,t)(u_0(y) - 1) dy\bigg|
 \cr 
& \leq C \int_{\R^n \backslash B_1(0)} |b(x-y,t)| dy\cr
& =  C \int_{\R^n \backslash B_1(x)} |b(z,t)| dz\cr
& \leq C \int_{\R^n \backslash B_{1/2}(0)} |b(z,t)| dz\cr
& \to 0
\end{align}
as $t \downto 0$ 
in view of \eqref{HKest},
which shows $u(\cdot,t)$ converges uniformly in the supremum norm to
$(1/c_0) \neq 1$  on $B_{1/2}(0)$ as $t \downto 0$,
which contradicts the fact that $u(\cdot,t)$ converges to $u_0$ in the
$L^2$ norm as $t \downto 0$.
Similarly, if $ \int_{\R^n} b  =  0 $, one shows $u(x,t) \to 0$  uniformly in the supremum norm
 on $B_{1/2}(0)$ as $t \downto 0$,  which contradicts the fact that $u(\cdot,t)$ converges to $u_0$ in the
$L^2$ norm as $t \downto 0$.

\section{Appendix}

We require a rather specific form of the standard interpolation inequalities (see for example \cite[Theorems
7.25--7.28]{GT}).  We have thus provided
proofs and precise statements of that which we need here in the appendix.

\begin{lemma}\label{L2est}
For any smooth function $\phi:B_1(0) \to \R$ we have
\begin{equation}
\int_{B_{1-s\si}} |\grad^{2s} \phi|^2 + |\grad^{2s-1} \phi|^2\leq c(s,\si) \int_{B_1} |\lap^s
\phi|^2 +  |\lap^{s-1}\phi|^2 +  \ldots + |\phi|^2  \label{inducts}
\end{equation}
\end{lemma}
for any $s \in \N$ and $1>\si>0$, as long as $1-s\si >0$. 
\begin{proof}
We show the inequality \eqref{inducts}  for arbitrary smooth $\phi:B_1(0) \to \R$  using induction.

Step 1: $L^2$-theory (see for example \cite[Theorem 1, Section 6.3.1]{Evans})
tells us that for an arbitrary smooth function $\phi:B_1 \to \R$
  \begin{equation}
\int_{B_{1-\si}} |\grad^{2} \phi|^2 + |\grad \phi|^2\leq c(\si)
\int_{B_1} |\lap \phi|^2 + |\phi|^2 
\end{equation}
 as required.

Inductive Step:
Let $\al = (\al_1, \ldots, \al_n)$ be a multi-index with 
$0\leq \al_i \leq 2s$ and $\sum_{i=1}^n \al_i = 2s$. 
We use the notation $\grad_{\al} \phi := \grad_{\al_1} \grad_{\al_2}
\ldots \grad_{\al_n} \phi.$
Assume that statement  \eqref{inducts} is true for some $s \in \N$.
Again, $L^2$ theory,  see for example \cite{Evans} Theorem
1, Section 6.3.1, tells us that
\begin{align*}
 \int_{B_{1-(s+1)\si}} &|\grad^{2} \grad_{\al} \phi|^2 + |\grad \grad_{\al} \phi|^2 
\\& \leq
a(s,\si) \int_{B_{1-s\si}} |\lap (\grad_{\al} \phi)|^2 + |\grad_{\al}
\phi|^2\cr
&= a(s,\si) \int_{B_{1-s\si}} |\grad_{\al} (\lap \phi)|^2 + |\grad_{\al}
\phi|^2\cr
& \leq a(s,\si) c(s,\si)  \int_{B_1} |\lap^s
(\lap \phi) |^2 +  |\lap^{s-1} ( \lap \phi)|^2 +  \ldots + |\lap
\phi|^2\cr 
& \ \ \ + a(s,\si) c(s,\si)  \int_{B_1} |\lap^s
\phi|^2 +  |\lap^{s-1}\phi|^2 +  \ldots + |\phi|^2 , \cr
& = \ti a (s,\si) \int_{B_1} |\lap^{s+1}
\phi|^2 +  |\lap^{s}\phi|^2 +  \ldots + |\phi|^2 
\end{align*}
where in the second last line (the  inequality) we have used the induction hypothesis
applied to the functions $\lap \phi$ and $\phi$.
By summing up over all possible $\al$ (the number of possible $\al$
is a constant depending on $n$ and $s$) we see that 
\begin{align*}
 \int_{B_{1-(s+1)\si}} |\grad^{2s +2} \phi|^2 + |\grad^{2s +1} \phi|^2
 & \leq c (s +1,\si) \int_{B_1} |\lap^{s+1}
\phi|^2 +  |\lap^{s}\phi|^2 +  \ldots + |\phi|^2 
\end{align*}
as required.

This completes the proof by induction.
\end{proof}

\begin{lemma}
Suppose $u\in C^\infty_{\text{loc}}(\R^n)$, $k\in\N$, $s>8$, and $\gamma$, $c_\gamma$ are as in \eqref{EQgamma}.
For any $\delta_0>0$ we have
\label{LMinterp1}
\begin{equation*}
\int_{\R^n} |\Delta^ku|^2\gamma^{s-4}
 \le \delta_0\rho^4 \int_{\R^n} |\Delta^{k+1}u|^2\gamma^s
   + \frac{c_{\delta_0}}{\rho^4}\int_{\R^n} |\Delta^{k-1}u|^2\gamma^{s-8}\,.
\end{equation*}
where $c_{\delta_0}$ is an absolute constant given by
\[
c_{\delta_0}
 = c_{\delta_0}(\delta_0,s,n)
 = 2^4(\delta_0^{-1} + 2^9s^4\cg^4)\,.
\]
\end{lemma}
\begin{proof}
Integrating by parts,
\begin{align*}
\int_{\R^n} |\Delta^ku|^2\gamma^{s-4}
 &= -\int_{\R^n} (\nabla_i\Delta^ku)(\nabla_i\Delta^{k-1}u)\gamma^{s-4}
\\&\qquad
    -(s-4)\int_{\R^n} (\Delta^ku)(\nabla_i\Delta^{k-1}u)(\nabla_i\gamma)\gamma^{s-5}
\\
 &=  \int_{\R^n} (\Delta^{k+1}u)(\Delta^{k-1}u)\gamma^{s-4}
\\*&\qquad
    + (s-4)\int_{\R^n} (\nabla_i\Delta^ku)(\Delta^{k-1}u)(\nabla_i\gamma)\gamma^{s-5}
\\*&\qquad
    -(s-4)\int_{\R^n} (\Delta^ku)(\nabla_i\Delta^{k-1}u)(\nabla_i\gamma)\gamma^{s-5}
\\
 &=  \int_{\R^n} (\Delta^{k+1}u)(\Delta^{k-1}u)\gamma^{s-4}
\\*&\qquad
    -2(s-4)\int_{\R^n} (\Delta^ku)(\nabla_i\Delta^{k-1}u)(\nabla_i\gamma)\gamma^{s-5}
\\*&\qquad
    - (s-4)\int_{\R^n} (\Delta^ku)(\Delta^{k-1}u)(\Delta\gamma)\gamma^{s-5}
\\*&\qquad
    - (s-4)(s-5)\int_{\R^n} (\Delta^ku)(\Delta^{k-1}u)|\nabla\gamma|^2\gamma^{s-6}\,.
\end{align*}
Estimating the right hand side with Young's inequality (estimate \eqref{EQcauchy}), we have for any $\delta_i>0$
\begin{align*}
\int_{\R^n}& |\Delta^ku|^2\gamma^{s-4}
\\
&\le
  \delta_1\rho^4\int_{\R^n} |\Delta^{k+1}u|^2\gamma^{s}
  + \frac{1}{4\delta_1\rho^4}\int_{\R^n} |\Delta^{k-1}u|^2\gamma^{s-8}
\\*&\qquad
 + \delta_2\int_{\R^n} |\Delta^ku|^2\gamma^{s-4}
  + \frac{c_\gamma^2(s-4)^2}{\delta_2\rho^2}\int_{\R^n} |\nabla\Delta^{k-1}u|^2\gamma^{s-6}
\\*&\qquad
 + \delta_3\int_{\R^n} |\Delta^ku|^2\gamma^{s-4}
  + \frac{c_\gamma^2(s-4)^2}{4\delta_3\rho^4}\int_{\R^n} |\Delta^{k-1}u|^2\gamma^{s-6}
\\*&\qquad
 + \delta_4\int_{\R^n} |\Delta^ku|^2\gamma^{s-4}
  + \frac{c_\gamma^4(s-4)^2(s-5)^2}{4\delta_4\rho^4}\int_{\R^n} |\Delta^{k-1}u|^2\gamma^{s-8}\,,
\end{align*}
which upon absorption gives
\begin{align}
\int_{\R^n}& |\Delta^ku|^2\gamma^{s-4}
\notag\\
&\le
  4\delta_1\rho^4\int_{\R^n} |\Delta^{k+1}u|^2\gamma^{s}
  + \frac{16c_\gamma^2(s-4)^2}{\rho^2}\int_{\R^n} |\nabla\Delta^{k-1}u|^2\gamma^{s-6}
\notag\\*&\qquad
  + \frac{1}{\rho^4}\bigg(
    \frac{1}{\delta_1}
  + 4c_\gamma^2(s-4)^2
  + 4c_\gamma^4(s-4)^2(s-5)^2
    \bigg)\int_{\R^n} |\Delta^{k-1}u|^2\gamma^{s-8}\,,
\label{EQlminterpeq1}
\end{align}
where we have chosen $\delta_2 = \delta_3 = \delta_4 = \frac{1}4$, and
we used $\ga^{s-6}(\cdot) \leq \ga^{s-8}(\cdot)$, which follows in
view of $0 \leq \ga \leq 1$ and $s-6>s-8\geq 1$. 
For the second term we integrate by parts and estimate using Young's inequality to obtain
\begin{align*}
\int_{\R^n} &|\nabla\Delta^{k-1}u|^2\gamma^{s-6}
\\
 &= 
    - \int_{\R^n} (\Delta^{k}u)(\Delta^{k-1}u)\gamma^{s-6}
    - (s-6)\int_{\R^n} (\Delta^{k-1}u)(\nabla_i\Delta^{k-1}u)(\nabla_i\gamma)\gamma^{s-7}
\\
 &\le \frac{\rho^2}{64\cg^2(s-4)^2}\int_{\R^n} |\Delta^{k}u|^2\gamma^{s-4}
     + \frac{16\cg^2(s-4)^2}{\rho^2}\int_{\R^n} |\Delta^{k-1}u|^2\gamma^{s-8}
\\*&\qquad
     + \frac{1}{2}\int_{\R^n} |\nabla\Delta^{k-1}u|^2\gamma^{s-6}
     + \frac{c_\gamma^2(s-6)^2}{2\rho^2}\int_{\R^n} |\Delta^{k-1}u|^2\gamma^{s-8}\,.
\end{align*}
Absorbing the third term on the right into the left yields
\begin{align*}
&\int_{\R^n} |\nabla\Delta^{k-1}u|^2\gamma^{s-6}
 \le \frac{\rho^2}{32\cg^2(s-4)^2}\int_{\R^n} |\Delta^{k}u|^2\gamma^{s-4}
\\*&\qquad
     + \frac{\cg^2}{\rho^2}
       \bigg(32(s-4)^2 + (s-6)^2\bigg)
       \int_{\R^n} |\Delta^{k-1}u|^2\gamma^{s-8}\,,
\intertext{and so}
\frac{16\cg^2(s-4)^2}{\rho^2}&\int_{\R^n} |\nabla\Delta^{k-1}u|^2\gamma^{s-6}
 \le \frac12\int_{\R^n} |\Delta^{k}u|^2\gamma^{s-4}
\\*&\qquad
     + \frac{16\cg^4(s-4)^2}{\rho^4}\bigg(32(s-4)^2
     + (s-6)^2\bigg)
       \int_{\R^n} |\Delta^{k-1}u|^2\gamma^{s-8}\,.
\end{align*}
Combining the above with \eqref{EQlminterpeq1} we have
\begin{align*}
\int_{\R^n}& |\Delta^ku|^2\gamma^{s-4}
\le
  4\delta_1\rho^4\int_{\R^n} |\Delta^{k+1}u|^2\gamma^{s}
  + \frac12\int_{\R^n} |\Delta^ku|^2\gamma^{s-4}
\notag\\*&\qquad
  + \frac{1}{\rho^4}\bigg(
       {16\cg^4(s-4)^2}\bigg(32(s-4)^2 + (s-6)^2\bigg)
\notag\\*&\qquad\qquad
  + \frac{1}{\delta_1}
  + 4c_\gamma^2(s-4)^2
  + 4c_\gamma^4(s-4)^2(s-5)^2
    \bigg)\int_{\R^n} |\Delta^{k-1}u|^2\gamma^{s-8}\,,
\end{align*}
Absorbing the second term on the right into the left, multiplying through by 2 and choosing $\delta_1 = \frac{\delta_0}{8}$ yields the estimate
\begin{equation*}
\int_{\R^n} |\Delta^ku|^2\gamma^{s-4}
 \le \delta_0\rho^4 \int_{\R^n} |\Delta^{k+1}u|^2\gamma^s
   + \frac{\tilde{c}_{\delta_0}}{\rho^4}\int_{\R^n} |\Delta^{k-1}u|^2\gamma^{s-8}\,
\end{equation*}
where
\begin{align*}
\tilde{c}_{\delta_0}
 &=
     32\cg^4(s-4)^2\left(32(s-4)^2 + (s-6)^2\right)
    + \frac{16}{\delta_0} + 8(s-4)^2\left(\cg^2 + \cg^4(s-5)^2\right) 
\\
 &= \frac{16}{\delta_0} + 8(s-4)^2\cg^2\left(1 + \cg^2(s-5)^2 + 4\cg^2\left(32(s-4)^2 + (s-6)^2\right)\right)
 \,.
\end{align*}
Since $s>8$ and $\cg\ge1$ we have
\begin{align*}
\tilde{c}_{\delta_0}
 &\le \frac{16}{\delta_0} + 8s^2\cg^2\left(1 + \cg^2s^2 + 4\cg^2\left(32s^2 + s^2\right)\right)
\\
 &\le \frac{2^4}{\delta_0} + 2^3s^4\cg^2\left(1 + \cg^2 + 2^2\cg^2\left(2^5 + 1\right)\right)
\\
 &\le 2^4(\delta_0^{-1} + 2^9s^4\cg^4)
\\
 &:= c_{\delta_0}\,,
\end{align*}
yielding the constant  claimed.
\end{proof}

\begin{corollary}
\label{LMinterp2}
Suppose $u\in C^\infty_{\text{loc}}(\R^n)$, $k\in\N$,  and $\gamma$, $c_\gamma$ are as in \eqref{EQgamma}.
For all $k\in\N$, $s > 4k$,
\begin{equation*}
\int_{\R^n} |\Delta^ku|^2\gamma^{s-4}
 \le \delta\rho^4 \int_{\R^n} |\Delta^{k+1}u|^2\gamma^s
   + \frac{c}{\rho^{4k-4}}\int_{\R^n} |\Delta u|^2\gamma^{s-4k}\,,
\end{equation*}
where $c(\de,s,n)< \infty$ is a constant depending on $\de,s,n$.
\end{corollary}
\begin{proof}
We shall proceed by induction in $k \in \N$. We wish to show that
\begin{equation}
E^k_{\gamma^{s-4}}(u)
 \le {\hat \delta}\rho^4 E^{k+1}_{\gamma^s}(u)
   + \frac{\hat c_{\hat \de,k,s} }{\rho^{4k-4}}E^{1}_{\gamma^{s-4k}}(u)
\,.
\label{EQindd2}
\end{equation}
is true for all $s > 4k$ for all $\hat \de > 0$, where
$\hat c_{\hat \de,k,s}= \hat c(\hat \de,s,k,n)$.
Let $k=1$.
Then \eqref{EQindd2} is true for all $s > 4k$ for all $\de>0$ trivially with the
choice of $ \hat c_{\hat \de,k,s}= 1$.

Assume \eqref{EQindd2} for some fixed $k \in \N$, and let $s >
4(k+1)$,$\de >0$  be arbitrary. Then $s > 8$ and
Lemma \ref{LMinterp1} gives the estimate
\begin{equation}
E^l_{\gamma^{s-4}}(u)
 \le
     \delta\rho^4 E^{l+1}_{\gamma^s}(u)
   + \frac{c_{\delta,l,s}}{\rho^4} E^{l-1}_{\gamma^{s-8}}(u)
\label{EQindd1}
\end{equation}
for any $l\in\N$, where $c_{\delta,l,s} = c(\de,l,s,n)$.

Using this inequality with $l = k+1 \in \N$ and then  \eqref{EQindd2} we have
\begin{align*}
E^{k+1}_{\gamma^{s-4}}(u)
 &\le
    \delta\rho^4 E^{k+2}_{\gamma^s}(u)
  + \frac{c_{\delta, {k+1},s } } {\rho^4} E^{k}_{\gamma^{s-8}}(u)
\\
 &\le
     \delta\rho^4 E^{k+2}_{\gamma^s}(u)
   + \hat{\delta} c_{\delta, k+1,s} E^{k+1}_{\gamma^{s-4}}(u)
   + \frac{\hat c_{\hat \delta, k,s-4} c_{\de,k+1,s}}{\rho^{4k}}E^{1}_{\gamma^{s-4-4k}}(u)
\,.
\end{align*}
where here we used the fact that $s > 4(k+1) = 4k +4$ implies that
$\ti s = s -4 > 4k$ and so \eqref{EQindd2}  is valid with $s$
replaced by $\ti s$.
Choosing
$\hat{\delta} c_{\delta,{k+1},s} = \frac12$ and absorbing we obtain
\begin{align*}
E^{k+1}_{\gamma^{s-4}}(u)
 &\le
     2\delta \rho^4 E^{k+2}_{\gamma^s}(u)
+ 2\frac{\hat c_{\hat \delta, k,s-4} c_{\de,k+1,s}  }{  \rho^{4k}  }E^{1}_{\gamma^{s-4-4k}}(u)
\,.
\end{align*}
which gives us the result, as $\de>0$ was arbitrary.
Note that we can ensure the constant only depends on $n,s$ and not $k$
by taking the supremum of the constants we obtained in this argument,
where this supremum is taken over all $4k<s$, $k \in \N$, for a fixed $s \in
\N$.
\end{proof}




\section*{Acknowledgements}
The first author would like to thank the University of Wollongong,
where part of this work was carried out.
The second author was partially supported by Alexander-von-Humboldt
fellowship 1137814 at the Otto-von-Guericke Universität Magdeburg and
by Australian Research Council Discovery Project grant DP120100097 at
the University of Wollongong. Further partial support toward both authors was provided by University of Wollongong Research Council 20124 Small Grant 22831024. They are grateful for their support.

\bibliographystyle{plain}

\end{document}